\documentclass[11pt]{article}
\usepackage{amsmath}
\usepackage{amsthm}
\usepackage{amsfonts}
\usepackage{amssymb}
\usepackage{enumerate}
\usepackage{dsfont}
\usepackage{graphicx}
\usepackage{epstopdf}
\usepackage{geometry}
\usepackage{stmaryrd}
\usepackage[colorlinks=true,linkcolor=blue,urlcolor=blue,citecolor=red, hyperfigures=false]{hyperref}

\let\inf\relax \DeclareMathOperator*\inf{\vphantom{p}inf}
\let\max\relax \DeclareMathOperator*\max{\vphantom{p}max}
\let\min\relax \DeclareMathOperator*\min{\vphantom{p}min}

\newcommand{\be}{\begin{equation}}
\newcommand{\ee}{\end{equation}}

\numberwithin{equation}{section}
\numberwithin{figure}{section}
\newtheorem{theorem}{Theorem}[section]
\theoremstyle{plain}

\newtheorem{definition}[theorem]{Definition}

\newtheorem{lemma}[theorem]{Lemma}
\newtheorem{notation}[theorem]{Notation}

\newtheorem{proposition}[theorem]{Proposition}
\newtheorem{remark}[theorem]{Remark}

\numberwithin{figure}{section}

\newcommand{\De}{{\Delta}}

\newcommand{\RR}{{\mathbb{R}}}
\newcommand{\DD}{{\mathbb{D}}}
\newcommand{\EE}{{\mathbb{E}}}
\newcommand{\NN}{{\mathbb{N}}}
\newcommand{\PP}{{\mathbb{P}}}

\newcommand{\CF}{{\mathcal{F}}}

\newcommand{\CH}{{\mathcal{H}}}

\newcommand{\CA}{{\mathcal{A}}}

\newcommand{\CS}{{\mathcal{S}}}

\newcommand{\CT}{{\mathcal{T}}}
\newcommand{\CB}{{\mathcal{B}}}

\newcommand{\indic}{{\mathds{1}}}
\newcommand{\comment}[1]{}

\def\p{\vskip4truept \noindent}

\newcommand{\tr}{{\!\!~^\top\!\!}}
\newcommand{\Leb}{{\rm Leb}}
\title{Zero-sum stopping games with asymmetric information.}
\author{Fabien Gensbittel\thanks{Toulouse School of Economics, University of Toulouse  Capitole, Toulouse, France \newline
 Manufacture des Tabacs, MF213, 21, All\'{e}e de Brienne 31015 Toulouse Cedex 6. \newline E-mails:\href{mailto:fabien.gensbittel@tse-fr.eu}{fabien.gensbittel@tse-fr.eu},\href{mailto:christine.gruen@tse-fr.eu}{christine.gruen@tse-fr.eu}}, Christine Gr\"{u}n$^*$}

\begin{document}
\maketitle

\begin{abstract}
We study a model of two-player, zero-sum, stopping games with asymmetric information. We assume that the payoff depends on two continuous-time Markov chains $(X_t), (Y_t)$, where $(X_t)$ is only observed by player 1 and $(Y_t)$ only by player 2, implying that the players have access to stopping times with respect to different filtrations.  We show the existence of a value in mixed stopping times and provide a variational characterization for the value as a function of the initial distribution of the Markov chains. We also prove a verification theorem for optimal stopping rules which allows to construct optimal stopping times. Finally we use our results to solve explicitly two generic examples.
\end{abstract}

\vspace{6cm}
\noindent\textbf{Acknowledgments :} 
The two authors gratefully acknowledge the support of the Agence Nationale de la Recherche, under grant ANR JEUDY, ANR-10-BLAN 0112.

\newpage
\setcounter{tocdepth}{1}

\section{Introduction}\label{sectionintroduction}

In this paper we consider a two player zero-sum stopping game with asymmetric information. The payoff depends on two independent continuous time Markov chains $(X_t, Y_t)_{t \geq 0}$ with finite state space $K\times L$, commonly known initial law $p\otimes q$ and infinitesimal generators $R= (R_{k,k'})_{k,k' \in K}$ and $Q=(Q_{\ell,\ell'})_{\ell,\ell' \in L}$. We assume that $X$ is only observed by player 1 while $Y$ is only observed by player 2. The fact that the game was not stopped up to some time gives to each player some additional information about the unknown state.  This is a crucial point as it implies that players have to take into account which information they generate about their private state when searching for optimal strategies. In consequence, our analysis is significantly different to that of classical stopping games.

We prove the existence of the value $V(p,q)$ in mixed stopping times allowing the players to randomize their stopping decision. Mixed stopping times have already been studied by Baxter and Chacon \cite{Baxter} and Meyer \cite{meyer} in a continuous time setting and applied to stopping games by Vieille and Touzi \cite{vieilletouzi} and Laraki and Solan \cite{larakisolan}. We also refer to the recent work of Shmaya and Solan \cite{shmayasolan} for a concise study of this type of stopping times. We work under the common assumptions on the payoffs used by Lepeltier and Maingueneau \cite{LepeltierMaingueneau}, in order to provide a variational characterization for $V$. Moreover, we show how this variational characterization can be applied to determine optimal strategies in the case of incomplete information on both sides. This result is new, since up to now similar characterizations are only known for the case of incomplete information on one side, and applies to a wide class of examples.

The variational characterization for $V$ that we provide can be seen as an extension (in a simple case) of the classical semi-harmonic characterization for stopping games of Markov processes with symmetric information (see e.g. Friedman \cite{friedman}, Eckstr\"om and Peskir \cite{eckstrom}) to models with asymmetric information. {It is reminiscent of the variational representation for the value of repeated games with asymmetric information given by Mertens and Zamir \cite{MZ} (see also  Rosenberg and Sorin \cite{rosenbergsorin} and Laraki \cite{laraki}), who solved explicitly several examples having a similar flavor than the examples we analyze. Our characterization is also equivalent to a first-order PDE with convexity constraints as introduced by Cardaliaguet in \cite{cardadiff,cardadouble}. Using appropriate PDEs with convexity constraints, the ideas of \cite{cardadiff,cardadouble} have been used to analyze particular classes of continuous-time games  by Cardaliaguet and Rainer \cite{cardastochdiff,carda12}, continuous time stopping games by Gr\"un \cite{grunstopping} and continuous time limit of repeated games by Gensbittel \cite{fabiencavu,fabiencovariance}}.

Most of the literature on dynamic games with asymmetric information deals with models where the payoff-relevant parameters of the game that are partially unknown (say information parameters) do not evolve over time. Some recent works focus on models of dynamic games with asymmetric information and evolving information parameters (see e.g. Renault \cite{Renault2006}, Neyman \cite{neyman}, Gensbittel and Renault \cite{GensbittelRenault}, Cardaliaguet et al. \cite{cardaetal}, Gensbittel \cite{gensbittel2013}). 
All these works consider dynamic games with current or terminal payoffs while in the present work we study the case of continuous-time stopping games with time-evolving information parameters.

The paper is structured as follows. First we give a description of the model and the main definitions. In the third section we establish existence and uniqueness of the value by a variational characterization  using PDE methods. We use this result in the following section to characterize optimal strategies for both players. In section 5 we present two examples where explicit expressions for the value as well as optimal strategies for both players are provided. The appendix collects auxiliary results and some technical proofs.

\section{Model}

\subsection{Notation}

For any topological space $E$, $\CB(E)$ denotes its Borel $\sigma$-algebra, $\De(E)$ denotes the set of Borel probability distributions on $E$ and  $\delta_x$ denotes the Dirac measure at $x\in E$. Finite sets are endowed with the discrete topology and Cartesian products with the product topology. If $E$ is finite, then $|E|$ denotes its cardinal and $\Delta(E)$ is identified with the unit simplex of $\RR^{E}$.  $\langle.,.\rangle$ and $|.|$ applied to vectors stand the usual scalar product and Euclidean norm while 
$\tr M$ stands for the transpose of a matrix $M$.

\subsection{The dynamics}

Let $K,L$ be two non-empty finite sets. 
We consider two independent continuous-time, homogeneous Markov chains $(X_t)_{t \geq 0}$ and $(Y_t)_{t \geq 0}$ with state space $K$ and $L$, initial laws $p \in \Delta(K)$, $q\in \Delta(L)$ and infinitesimal generators $R= (R_{k,k'})_{k,k' \in K}$ and $Q=(Q_{\ell,\ell'})_{\ell,\ell' \in L}$ respectively. $R_{k,k'}$ represents as usual the  jump intensity  of the process $X$ from state $k$ to state $k'$ when $k'\neq k$ and $R_{k,k}=- \sum_{k'\neq k} R_{k,k'}$.

We denote $\PP^X_{p}$ the law of the process $X$ defined on the canonical space of $K$-valued c\`{a}dl\`{a}g trajectories $\Omega_X=\DD([0,\infty), K)$
and $\PP^Y_{q}$ the law of the process $Y$ defined on the space $\Omega_Y=\DD([0,\infty),L)$.
Furthermore, let us define 
\[(\Omega,\CA;\PP_{p,q}):=(\Omega_X\times \Omega_Y, \CF^X_{\infty}\otimes \CF^Y_\infty, \PP^X_{p}\otimes \PP^Y_{q}).\]
We will identify the  right-continuous filtrations $\CF^X$ and $\CF^Y$ (see Theorem 26 p. 304 in \cite{Bremaud})  as filtrations defined on $\Omega$ as well as $\CF^X_{\infty}$-measurable random variables on $\Omega$ as $\CF^X_{\infty}$-measurable variables defined on $\Omega_X$ (and similarly for $Y$).

We consider a zero-sum stopping game, where player $1$ observes the trajectory of $X$, while player $2$ observes the trajectory of $Y$. 
So according to their information the dynamics of the game for player 1 are basically given by
\[\left(\PP_{p,q}[X_t=k|\mathcal{F}^X_s]\right)_{k \in K}= e^{(t-s)\tr R} \delta_{X_s},\quad \left(\PP_{p,q}[Y_t=\ell|\mathcal{F}^X_\infty]\right)_{\ell \in L}= e^{t \tr Q} q,\]
while for player 2 they are given by
\[\quad \left(\PP_{p,q}[X_t=k|\mathcal{F}^Y_\infty]\right)_{k \in K}= e^{t \tr R} p, \quad \left(\PP_{p,q}[Y_t=\ell | \mathcal{F}^Y_s]\right)_{\ell \in L}= e^{(t-s)\tr Q}\delta_{Y_s},\]
where we use the independence of $X$ and $Y$.
\p

\subsection{Mixed stopping times and payoff function}

Let $\CT^X$ denote the set of $\CF^{X}$ stopping times and $\CT^Y$ denote the set of $\CF^{Y}$ stopping times. 
\begin{definition}\label{mixedstoppingtimes} \
A mixed stopping time of the filtration $\CF^X$ on $\Omega$ is an $\CF^X_\infty \otimes \CB([0,1])$-measurable map $\mu$ defined on $\Omega\times [0,1]$ such that $\mu(.,u)\in\CT^{X}$ for all $u \in [0,1]$. We denote $\CT^X_m$ the set of mixed $\CF^{X}$-stopping times.

The definition mixed stopping times of the filtration $\CF^Y$ is similar.  We denote $\CT^Y_m$ the set of mixed $\CF^{Y}$-stopping times.

 A random time is  a $\CB([0,1])$-measurable map $\mu:[0,1]\rightarrow[0,+\infty]$. The set of random times is denoted $\CT^\emptyset_m$.
\end{definition}
Let $r>0$ denote a fix discount rate and $f \geq h$ two real-valued functions defined on $K \times L$. The players choose mixed stopping times $\mu(\omega,u)\in\CT^X_m$ and $\nu(\omega,v)\in \CT^Y_m$ respectively, in order to maximize (resp. minimize) the expected payoff:
\p
\be \label{exppayoff} 
\mathbb{E}_{p,q}\left [\int_0^1 \int_0^1  J(\mu,\nu)(\omega,u,v) du dv \right],
\ee
where
\be \label{payoff} 
J(\mu,\nu)(\omega,u,v):=e^{-r \nu} f(X_{\nu},Y_{\nu}) \indic_{\nu <\mu}+e^{-r \mu} h(X_{\mu},Y_{\mu}) \indic_{\mu \leq\nu},
\ee
{with the convention that $J(\mu,\nu)(\omega,u,v)=0$ on $\{\mu=\nu=+\infty\}$.} Furthermore we set
\[ \bar{J}(\mu,\nu):=\int_0^1 \int_0^1  J(\mu,\nu)(\omega,u,v) du dv.\]
\p
The upper value of the game is defined by
\be
 V^+(p,q):= \inf_{\nu \in  \CT^Y_m}\; \sup_{\mu \in  \CT^X_m}\; \mathbb{E}_{p,q}\left [ \bar{J}(\mu,\nu) \right],
\ee
the lower value by
\be V^-(p,q):=  \sup_{\mu \in  \CT^X_m}\; \inf_{\nu \in  \CT^Y_m}\;  \mathbb{E}_{p,q}\left [  \bar{J}(\mu,\nu) \right],
\ee
where by definition $V^-(p,q)\leq V^+(p,q)$. When there is equality, we say that the game has a value $V:=V^-=V^+$.

{In the above expressions, the payoff is an expectation with respect to $(\omega,u,v)$, and we distinguish $\omega$, which is the trajectory of the process $(X,Y)$ and represents the exogenous randomness of the model, from $u$ and $v$, which are strategic randomizations introduced by the players.}

{Despite the apparent asymmetry in the payoff function \eqref{payoff}, our model is symmetric in the following sense: If we define the modified payoff 
\[J'(\mu,\nu)(\omega,u,v):=e^{-r \nu} f(X_{\nu},Y_{\nu}) \indic_{\nu \leq\mu}+e^{-r \mu} h(X_{\mu},Y_{\mu}) \indic_{\mu <\nu},\]
then $J' \geq J$,  $\lim_n J(\mu+\frac{1}{n},\nu) = J'(\mu,\nu)$ and $\lim_n J'(\mu,\nu+\frac{1}{n})=J(\mu,\nu)$, so that by bounded convergence, the upper and lower values of the games with respective payoffs $J$ and $J'$ coincide (see Lemma 5 in \cite{LepeltierMaingueneau} for more details). Therefore, we shall state and prove most of our results only on one side (properties of the upper value in section \ref{sectionresult}, construction of an optimal stopping time for player $1$ in section \ref{sectionverif}), as the corresponding result on the other side (properties of the lower value, optimal stopping time for player $2$) can be obtained by exchanging the roles of the players.}

Let us comment the notion of mixed (or randomized) stopping times. Our definition corresponds to the classical notion of mixed strategy in a game: for each player, a mixed strategy is a probability distribution over his possible pure (i.e. standard) strategies. 
Following  Aumann \cite{aumann}, a natural way to define mixed strategies when the set of pure strategies is a set of measurable maps but has no simple measurable structure (as $\CT^X$ and $\CT^Y$ in our model), is to introduce an auxiliary probability space which is used as a randomization device by the player, allowing him to choose at random a pure strategy. 
With such a definition, one has only to require the mixed strategy to be jointly measurable for the expected payoff to be well-defined. 
We refer the reader to \cite{shmayasolan} for a discussion of the different equivalent definitions of randomized stopping times. From an analytic point of view, the set of randomized stopping times $\CT^X_m$ is the closed convex hull for the weak topology (which may coincide with the closure in some cases but not in our model\footnote{Indeed, if $X_0=0$, for sufficiently small $\alpha>0$, it is not possible to obtain the linear form $\frac{1}{2}(\EE_{\delta_0}[Z_0+Z_{\alpha}])$ as a limit of classical stopping times. This is due to the fact that the event $\{ \forall s\in [0,\alpha], \, X_s=0\}$ has probability close to $1$ for small $\alpha$, implying that no stopping time in $\CT^X$ can be equal to $0$ and $\alpha$ with probabilities close to $1/2$.}) of the set of stopping times $\tau \in \CT^X$ seen as linear forms $\EE[Z_\tau]$ acting on the class of bounded $\CF^X_\infty$-measurable continuous processes $Z$, and it was proved to be compact in \cite{Baxter} and \cite{meyer}. We do not use such topological properties of mixed stopping times and our existence result is not based on some general minmax theorem, but convexity plays a central role in our analysis, as can be seen in the statement of Theorem \ref{thmextremal}.



It is well known that in games where both players have the same information the value exists under fairly general assumptions, when the players are only allowed to choose stopping times adapted to their common information. The existence of the value implies that there is no loss for either of the players if they give their adversary the advantage of playing second. For example in $V^+(p,q)$ the maximizer has an information advantage by knowing exactly which strategies he is facing. In games with incomplete information, existence of a value in non-randomized strategies is in general not true, as shown in the first example in section \ref{sectionexample}. The intuitive explanation is that the first player would, by choosing non randomized strategies, reveal too much information to the second player about the  process $X$, which is not observed by player 2. In contrast, using randomized strategies allow each player to manipulate the beliefs of his opponent. This translates into the fact that the value is concave in $p$ and  convex in $q$ (see Lemma \ref{propertiesV}) and, as can be seen in the examples of section \ref{sectionexample}, in general non-linear
with respect to the initial distributions $(p,q)\in\Delta(K)\times\Delta(L)$.
\section{Existence and characterization of the value}\label{sectionresult}
\subsection{Result}
Our first result is the existence of the value together with a variational characterization, which is a first-order PDE with convexity constraints. These constraints are expressed using the notion of extreme points as in \cite{carda12,laraki,rosenbergsorin}, {but could be equivalently written in terms of dual viscosity solution as in \cite{cardadiff} or as a viscosity solution of an equation with a double obstacle as in \cite{cardadouble}. The main advantage of our formulation is that the corresponding comparison principle (Theorem \ref{comparisonthm2}) is a short an easy adaptation of the one given by Mertens and Zamir in \cite{MZ}}. 
\begin{definition}
A function $g:\Delta(K)\times\Delta(L) \rightarrow \mathbb{R}$ is said to be a saddle function if it is concave with respect to $p\in \Delta(K)$ and convex with respect to $q \in \Delta(L)$. \\
For any $q\in\Delta(L)$, the set of extreme points $Ext(g(.,q))$ is defined as the set of all $p\in\Delta(K)$ such that 
\[ (p,g(p,q))=\lambda(p_1,g(p_1,q))+(1-\lambda)(p_2,g(p_2,q)) \]
with $\lambda\in (0,1)$ and $p_1,p_2\in\Delta(K)$ implies $p_1=p_2=p$.
The  set of extreme points $Ext(g(p,.))$ for any $p\in\Delta(K)$ is defined in a similar way.
\end{definition}
\begin{remark}
One can check easily that $Ext(g(.,q))$ is the set of $p\in \Delta(K)$ such that $(p,g(p,q))$ is an extreme point (in the usual sense, see Definition \ref{extremepoint}) of the hypograph of $g(.,q)$ defined by $\{(p',t)\in \Delta(K)\times \RR \, |\, t \leq g(p',q)\}$. That $p \in Ext(g(.,q))$ means that $g(.,q)$ is strictly concave at $p$, i.e. not affine in any non-trivial segment containing $p$. 
Note also that this definition implies that all the Dirac masses $\delta_k$ for $k \in K$ always belong to  $Ext(g(.,q))$.
Similarly properties hold for $Ext(g(p,.))$.
\end{remark}

Let $f,h$ be extended linearly on the set $\Delta(K)$, i.e.
\[ \forall (p,q)\in \Delta(K)\times \Delta(L), \; f(p,q):=\sum_{(k,\ell) \in K\times L} p_k q_\ell f(k,\ell),\;\; h(p,q):=\sum_{(k,\ell) \in K\times L} p_k q_\ell h(k,\ell).\]

\begin{theorem}\label{thmextremal}
For all $(p,q) \in \Delta(K)\times \Delta(L)$, the game has a value 
\[V(p,q):= V^+(p,q)= V^-(p,q),\]
and $V$ is the unique Lipschitz saddle function on $\Delta(K)\times \Delta(L)$ such that:\\
(Subsolution) $\forall q\in \Delta(L), \;  \forall p \in Ext(V(.,q))$, 
\begin{multline}\label{subpoint} 
\max\{ \min\{ rV(p,q) -\vec{D}_1V(p,q; \tr Rp) - \vec{D}_2V(p,q ; \tr Q q ); V(p,q)- h(p,q)\};\\
V(p,q)-f(p,q) \} \leq 0, 
\end{multline}
(Supersolution) $\forall p\in \Delta(K), \; \forall q \in Ext(V(p,.))$,
\begin{multline}\label{superpoint}
 \max\{ \min\{ rV(p,q) -\vec{D}_1V(p,q; \tr Rp) - \vec{D}_2V(p,q ; \tr Q q ); V(p,q)- h(p,q)\};\\
 V(p,q)-f(p,q) \} \geq 0, 
\end{multline}
where $\vec{D}_1V(p,q ;\xi)$ and $\vec{D}_2V(p,q ;\zeta)$ denote the directional derivatives of $V$ at $(p,q)$ with respect to the first and second variables in the directions $\xi$ and $\zeta$ respectively.
\end{theorem}

%

Let us comment \eqref{subpoint} and \eqref{superpoint}. If $p$ is an extreme point, i.e. $p \in Ext(V(.,q))$, \eqref{subpoint} is a standard subsolution property for the obstacle PDE,
\begin{multline}\label{point}
 \max\{ \min\{ rV(p,q) -\vec{D}_1V(p,q; \tr Rp) - \vec{D}_2V(p,q ; \tr Q q ); V(p,q)- h(p,q)\};\\
 V(p,q)-f(p,q) \} =0, 
\end{multline}
while if  $q$ is an extreme point, i.e. $q \in Ext(V(p,\cdot))$, \eqref{superpoint} is the supersolution property for this PDE. We note that because $V$ is a saddle function, it is reasonable to state the equation in the strong form since we can define directional derivatives. This allows us also to derive the comparison principle in section \ref{sectioncomparison} with classical tools.\\

Moreover the game where neither player observes $(X_t,Y_t)$ corresponds to a game with deterministic dynamic on $\Delta(K)\times\Delta(L)$ given by the marginal distribution of $(X_t,Y_t)$, i.e. $(e^{t\tr R} p,e^{t\tr Q} q)$ for some initial values $(p,q)\in\Delta(K)\times\Delta(L)$. It is well known that its value {denoted $S(p,q)$} is characterized as the unique viscosity solution to
{ 
\begin{multline}\label{pointstandard}
\max\{ \min\{ rS(p,q) -\langle \nabla_p S(p,q), \tr Rp\rangle -\langle \nabla_q S(p,q), \tr Qp\rangle; S(p,q)- h(p,q)\};\\
 S(p,q)-f(p,q) \} =0, 
\end{multline}
}
which is  \eqref{point} with a weaker a priori regularity. So loosely speaking, for extreme points of the value function of the game with incomplete information, $V(p,q)$ solves the same variational inequalities as  {$S(p,q)$}. However, the set of extreme points might be very small (but it contains at least all the Dirac masses) and our characterization also requires $V$ to be a saddle function. It follows that the value function $V(p,q)$ significantly differs from the solution to  \eqref{pointstandard} in general, as seen in the second example in section \ref{sectionexample} where the set of extreme points of $V$ contains only one point which is not a Dirac mass (see Remark \ref{solutionblind}).

{The rest of this section is devoted the the proof of Theorem \ref{thmextremal}. We adapt the techniques in Cardaliaguet \cite{cardadiff} to the context of stopping games.  The main structure of our proof is similar and goes as follows:
\begin{itemize}
\item We first establish in Lemma \ref{propertiesV} that $V^+$ is a Lipschitz saddle function (the convexity properties of values of games with incomplete information hold in very general models, see e.g. chapter 2 in \cite{Sorin}) 
\item We prove in Lemma \ref{lemmaminmax} that the concave conjugate $V^{+,*}$ of $V^+$ with respect to the first variable admits an alternative formulation, which is actually the lower value of the dual game introduced by De Meyer \cite{demeyerdual}. Using this result, we prove in Proposition \ref{DPPbothsides} that $V^{+,*}$ satisfies a super-dynamic programming principle.
\item As usual in the theory of differential games (see e.g. the seminal paper of Evans and Souganidis \cite{EvansSouganidis}), we deduce from the super-dynamic programming principle, that $V^{+,*}$ satisfies a variational inequality (Proposition \ref{V*supersolbothsides2}). Using duality, we prove in Proposition \ref{Vsubsolbothsides2} that this variational inequality implies that $V^+$ satisfies the subsolution property \eqref{subpoint}. 
\item By symmetry of the model, $V^-$ satisfies the supersolution property  \eqref{superpoint}. 
\item We conclude that the two preceding properties imply $V^-=V^+$ by proving a comparison principle (Theorem \ref{comparisonthm2}).
\end{itemize}}
\subsection{Properties of $V^+$ and its concave conjugate}
First we note the following facts.
\begin{remark}\label{purifystoppingtimes}
$V^+$ can be estimated from above by setting $\nu=0$ and taking into consideration that the obstacles satisfy $h\leq f$, we deduce that $V^+(p,q)\leq f(p,q)$. Similarly, we have $h(p,q) \leq V^+(p,q)$.
{On the other hand, we can replace the supremum over $\CT^X_m$ by a supremum over $\CT^X$ in the definition of $V^+$,  
i.e.
\be
\begin{array}{l}
 V^+(p,q)= \inf_{\nu \in  \CT^Y_m}\; \sup_{\mu \in \CT^X}\;  \EE_{p,q}[ \bar J (\mu,\nu)].
\end{array}
\ee
Indeed, using Fubini's theorem, for any $\nu\in \CT^Y_m$,  we have 
\[ \sup_{\mu \in  \CT^X_m}\EE_{p,q}[ \bar{J}(\mu,\nu) ]= \sup_{\mu \in  \CT^X_m}\int_0^1 \EE_{p,q}[ \int_0^1 J(\mu, \nu)dv] du \leq \sup_{\mu \in  \CT^X_m}\sup_{u \in [0,1]}  \EE_{p,q}[ \int_0^1 J(\mu(.,u), \nu)dv], \]
and since for all $\mu \in  \CT^X_m$ and all $u\in [0,1]$, $\mu(.,u)$ is a stopping time in $\CT^X$, we deduce that 
\[\sup_{\mu \in \CT^X_m}\;  \bar J (\mu,\nu)\leq \sup_{\mu \in \CT^X}\;  \bar J (\mu,\nu),\]
which proves the result since $\CT^X \subset \CT^X_m$.}
\end{remark}

{We summarize the properties of $V^+$ in the following lemma.}
\begin{lemma} \label{propertiesV}
{$V^+$ is a Lipschitz saddle function. }
\end{lemma}
\begin{proof}
\p
Note that for any $\mu \in  \CT^X_m,\nu \in  \CT^Y_m$, it holds by conditioning 
\be\label{auseinander}
\begin{array}{rcl}
\mathbb{E}_{p,q}\left [  \bar{J}(\mu,\nu) \right]&=&\sum_{k\in K, \ell\in L} \PP[X_0=k]\ \PP[Y_0=\ell]\ \ \mathbb{E}_{p,q} \left [ \bar{J}(\mu,\nu) | X_0=k, Y_0=\ell\right]\\
\ \\
&=&\sum_{k\in K, \ell\in L} p_k\  q_\ell \mathbb{E}_{\delta_k,\delta_\ell} \left [\bar{J}(\mu,\nu) \right],
\end{array}
 \ee
$\delta_k$, $\delta_\ell$ denoting the Dirac masses at $k,\ell$ identified with the $k$-th, $\ell$-th vectors in the canonical bases of $\RR^K$ and $\RR^L$ respectively.
\p
In order to show the Lipschitz continuity, let $p,p'\in\Delta(K), q,q'\in\Delta(L)$ such that $0<V^+(p,q)-V^+(p',q')$. Choosing $\nu^*\in \CT^Y_m$ $\varepsilon$-optimal for $V^+(p',q')$ and $\mu^*\in \CT^Y_m$ $\varepsilon$-optimal for $\sup_{\mu \in \CT^X_m}\;  \mathbb{E}_{p,q}\left [  \bar{J}(\mu,\nu^*) \right]$ we have
\be
0<V^+(p,q)-V^+(p',q')\leq \mathbb{E}_{p,q}\left [  \bar{J}(\mu^*,\nu^*) \right]-\mathbb{E}_{p',q'}\left [  \bar{J}(\mu^*,\nu^*) \right]+2\varepsilon
\ee
for $\varepsilon$ arbitrarily small. The Lipschitz continuity follows then immediately by (\ref{auseinander}).
\p
Furthermore we claim that:
\be \label{eqVconcave}
\begin{array}{rcl}
V^+(p,q)&=& \inf_{\nu \in  \CT^Y_m}\; \sup_{\mu \in  \CT^X}\;  \mathbb{E}_{p,q}\left [  \bar{J}(\mu,\nu) \right]\\
& = & \inf_{\nu \in \CT^Y_m}\;  \sum_{k\in K}\ p_k \ \left(\sup_{\mu \in  \CT^X} \mathbb{E}_{\delta_k,q} \left [\bar{J}(\mu,\nu)\right]\right).
\end{array}
\ee
Indeed, $V^+(p,q)$ is clearly less or equal than the second line in the above equation. To prove the reverse inequality, for any $\nu\in \CT^Y_m$, and any $k \in K$, let $\mu^k$ be some  $\varepsilon$-optimal stopping time for the problem $\sup_{\mu \in  \CT^X} \mathbb{E}_{\delta_k,q} \left [\bar{J}(\mu,\nu)\right]$. Setting $\mu= \sum_{k\in K}\indic_{X_0 =k}\mu^k$ we note that
\[ \sum_{k\in K}\ p_k \ \mathbb{E}_{\delta_k,q} \left [\bar{J}(\mu^k,\nu)\right] = \mathbb{E}_{p,q} \left [\bar{J}(\mu,\nu)\right]\]
and  \eqref{eqVconcave} follows by sending $\varepsilon$ to zero.
\p
We deduce from \eqref{eqVconcave} that $p \rightarrow V^+(p,q)$ is concave as an infimum of affine functions. The convexity in $q$ follows by the classical splitting method. Let $q_1,q_2,q\in\Delta(L)$, $\lambda\in (0,1)$ such that
\[q=\lambda q_1+(1-\lambda) q_2.\]
We choose $\nu_1 \in \CT^Y_m$, $\nu_2 \in \CT^Y_m$ $\varepsilon$-optimal for $V^+(p,q_1)$ and  $V^+(p,q_2)$ respectively. Then we will construct  $\nu\in \CT^Y_m$ such that 
\be\label{equ2}
\mathbb{E}_{p,q}\left [  \bar{J}(\mu,\nu) \right]=\lambda \mathbb{E}_{p,q_1}\left [  \bar{J}(\mu,\nu_1) \right]+(1-\lambda) \mathbb{E}_{p,q_2}\left [  \bar{J}(\mu,\nu_2) \right].
\ee
The intuition of the construction is the following: At time $t=0$, player $2$, knowing $Y_0$, can choose at random a decision $d \in \{1,2\}$  such that the conditional law of $Y_0$ given that $d=1$ is $q_1$ and the conditional law of $Y_0$ given $d=2$ is $q_2$. He will then play $\nu_1$ if $d=1$ and $\nu_2$ when $d=2$. More precisely, we set
\begin{equation}
\nu(\omega,u)= \sum_{\ell=1}^L  \indic_{Y_0=\ell} \left(\indic_{u\in[0,\frac{\lambda (q_1)_\ell}{q_\ell}]} \nu_1(\omega,\frac{q_\ell}{\lambda (q_1)_\ell} u)+\indic_{u\in(\frac{\lambda (q_1)_\ell}{q_\ell},1]} \nu_2(\omega,\frac{q_\ell u-\lambda (q_1)_\ell}{(1-\lambda) (q_2)_\ell} )\right)
\end{equation}
By the definition of $\mu$ the probability to choose $\nu_1$ given that $Y_0=\ell$ is $\frac{\lambda (q_1)_\ell}{q_\ell}$ whenever $q_\ell>0$ and the probability to choose $\nu_2$ is $\frac{(1-\lambda) (q_2)_\ell}{q_\ell}$. It follows that
\begin{align*}
\mathbb{E}_{p,q}\left [  \bar{J}( \mu,\nu) \right]&= \mathbb{E}_{p,q}\left [  \int_0^1 {J}(\mu,{\nu}(.,u)) du \right] \\
&= \mathbb{E}_{p,q}\bigg [  \sum_{\ell \in L} \indic_{Y_0=\ell}  \bigg( \int_0^{\frac{\lambda (q_1)_\ell}{q_\ell}} {J}(\mu,{\nu_1}(.,\frac{q_\ell}{\lambda (q_1)_\ell} u))\\
&\ \ \ \ \ \ \ \ \ \ \ \ \ \ \ \ \ \ \ \ \ \ \ \ \ \ \ +  \int_{\frac{\lambda (q_1)_\ell}{q_\ell}}^{1} {J}(\mu,{\nu_2}(.,\frac{q_\ell u-\lambda (q_1)_\ell}{(1-\lambda) (q_2)_\ell})) \bigg)du \bigg]      \\
&=  \mathbb{E}_{p,q}\left [ \sum_{\ell \in L} \indic_{Y_0=\ell}\left(\frac{\lambda (q_1)_\ell}{q_\ell}\bar{J}(\mu,\nu_1) + \frac{(1-\lambda) (q_2)_\ell}{q_\ell} \bar{J}(\mu,\nu_2) \right)   \right]       \\
&= \lambda \mathbb{E}_{p,q_1}\left [  \bar{J}(\mu,\nu_1) \right]+(1-\lambda) \mathbb{E}_{p,q_2}\left [  \bar{J}(\mu,\nu_2) \right].
\end{align*}
Maximizing (\ref{equ2}) over $\mu \in  \CT^X_m$ yields then, using the $\varepsilon$ optimality of $\nu_1$ and $\nu_2$,
\[ V^+(p,q)\leq   \lambda V^+(p,q_1)+ (1-\lambda )V^+(p,q_2)+\varepsilon \]
and the convexity in $q$ follows since $\varepsilon$ can be chosen arbitrarily small.
\end{proof}
\p

Next we define the concave conjugate in $p$ of $V^{+}$  as 
\[\forall x \in \RR^K, q\in \Delta(L), \;  V^{+,*}(x,q):= \inf_{p\in \Delta(K)}\{ \langle x,p \rangle - V^+(p,q)\}\]
As $h\leq V^+ \leq f$, it follows that:
\be
\begin{array}{l}
 f^*(x,q) \leq V^{*,+}(x,q) \leq h^*(x,q),\\
\end{array}
\ee
where the functions $h^*, f^*$ are defined as
\be
\begin{array}{l}
h^*(x,q):= \inf_{p\in \Delta(K)} \{\langle x,p \rangle - h(p,q)\},\;  f^*(x, q):=\inf_{p \in\Delta(K)} \{\langle x,p \rangle - f(p,q)\}.
\end{array}
\ee
\p
The next lemma provides an alternative formulation, {which is actually the opposite of the upper value of a modified version of the game, called the dual game (see \cite{demeyerdual}), where player $1$ chooses privately the initial state $X_0=k$ at the beginning of the game, and has to pay a cost $x^k$. }
\begin{lemma}\label{lemmaminmax}
We have the following, alternative representation: 
\be\label{rewrite1}
\forall x \in \RR^K, q\in \Delta(L), \;V^{+,*}(x,q)=\sup_{\nu \in \CT^Y_m}\; \; \inf_{\mu \in \CT^X}\; \inf_{p \in \Delta(K)} \big(\langle x,p \rangle - \EE_{p,q}[\bar{J}(\mu,\nu)] 
\big)\ee
\end{lemma}

\begin{proof}
Using remark \ref{purifystoppingtimes}, we have 
\begin{align*}
V^{+,*}(x,q)=   \inf_{p \in \Delta(K)}\; \sup_{\nu \in \CT^Y_m}\; \inf_{\mu \in \CT^X}\; \big(\langle x,p \rangle - \EE_{p,q}[\bar{J}(\mu,\nu)]\big) .
\end{align*}
Then, we will apply Fan's minmax theorem (see \cite{fan}) to deduce that: 
\begin{align*}
V^{+,*}(x,q)&= \inf_{p \in \Delta(K)}\; \sup_{\nu \in \CT^Y_m} \; \inf_{\mu \in \CT^X}\; \big(\langle x,p \rangle - \EE_{p,q}[\bar{J}(\mu,\nu)]\big)\\ 
&=  \sup_{\nu \in \CT^Y_m}\; \inf_{p \in \Delta(K)}\; \inf_{\mu \in \CT^X}\;\big( \langle x,p \rangle - \EE_{p,q}[\bar{J}(\mu,\nu)]\big). 
\end{align*}
In order to apply Fan's minmax theorem, we have to check that the function
\[ (p,\nu)\in \Delta(K)\times \CT^Y_m \rightarrow  \inf_{\mu \in \CT^X} \; \big( \langle  x,p \rangle - \EE_{p,q}[\bar{J}(\mu,\nu)]\big) \]
is concave-like with respect to $\nu\in \CT^Y_m$ and affine (hence continuous) with respect to $p$ in the compact convex set $\Delta(K)$.
\p
To prove the concave-like property, given $\nu_1,\nu_2 \in \CT^Y_m$, and $\lambda \in (0,1)$, we define the mixed stopping time $\nu$ by
\[ \nu (\omega,u) = \nu_1(\omega, \frac{u}{\lambda}) \indic_{u \in [0,\lambda)} + \nu_2(\omega,\frac{u-\lambda}{1-\lambda}) \indic_{u \in [\lambda,1]}.\]
A simple change of variables gives:
\begin{align*}
 \inf_{\mu \in \CT^X} \big(\langle x,p \rangle - &\EE_{p,q}[ \bar{J}(\mu,\nu) ] \big) \\
 &\geq 
  \lambda \inf_{\mu \in \CT^X} \big( \langle x,p \rangle - \EE_{p,q}[ \bar{J}(\mu,\nu_1)] \big) + (1-\lambda) \inf_{\mu \in \CT^X} \big( \langle x,p \rangle - \EE_{p,q}[ \bar{J}(\mu,\nu_2)] \big),
\end{align*}
which is exactly the the concave-like property we need to apply Fan's theorem.
\p
The second property follows from the relation:
\[\inf_{\mu \in \CT^X} \;  \big(\langle  x,p \rangle - \EE_{p,q}[\bar{J}(\mu,\nu)]\big) = \langle x,p \rangle - \sum_{k\in K} p^k \sup_{\mu \in \CT^X} \EE_{\delta_k,q}[\bar{J}(\mu,\nu)]. \]
This last equation is proved in the same way as equation \eqref{eqVconcave} in Lemma \ref{propertiesV}.
\end{proof}

\subsection{Dynamic Programming for $V^{+,*}$}

We will now prove a dynamic programming inequality for $V^{+,*}$. {This result is very intuitive, it reflects the fact that it is sub-optimal for player 2 in the dual game to wait and start the game only at time $\varepsilon$ without using the information conveyed by the process $Y$ on the time interval $[0,\varepsilon)$.}
\begin{proposition}\label{DPPbothsides}
For all $\varepsilon>0$
\begin{align}\label{DPP}
 V^{+,*}(x,q) \geq \; \inf_{t \in [0,\varepsilon]}\; \left(e^{-r t} h^*(x_t,q_t )\indic_{t< \varepsilon} + e^{-r \varepsilon}V^{+,*}(x_\varepsilon,q_\varepsilon) \indic_{t = \varepsilon}\right),
\end{align}
where the dynamic $(x_t,q_t)$ is given by
\begin{equation*}
\forall t\geq 0, \; x_t=x+\int_0^t  (rI-R)x_s ds\ \;\;\;\;\textnormal{\textit{and} }\ \;\;\;\;\ q_t=q+\int_0^t \tr Q  q_s ds.
\end{equation*}
\end{proposition}

\begin{proof}
In order to prove the dynamic programming inequalities, we need to recall the definition of the shift operator $\theta^X$ on $\Omega_X$. 
For all $t \geq 0$, the map $\theta^X_t: \Omega_X \rightarrow \Omega_X$ is defined by 
\[ \forall s\geq 0, \;\theta^X_t(\omega_X) (s) = \omega_X(s+t) .\]
The shift operator  $\theta^Y$ on $\Omega_Y$ is defined similarly.

 Given $\varepsilon>0$, we consider the family $\CT^Y_{m,\varepsilon}$ of mixed stopping times $\nu \in \CT^Y_m$ such that there exists a mixed stopping time $\nu' \in \CT^Y_m$ and $\nu(\omega,v)= \varepsilon+ \nu'(\theta^Y_{\varepsilon}(\omega_Y),v)$.
\p
As $\CT^Y_{m,\varepsilon} \subset \CT^Y_m$, we have
\begin{multline}\label{ineqDPP}
V^{+,*}(x,q) \geq \sup_{\nu \in \CT^Y_{m,\varepsilon}}\;\inf_{\mu \in \CT^X}\; \inf_{p \in \Delta(K)} \Big( \langle x,p\rangle - \EE_{p,q}\Big[ e^{-r \mu}h(X_\mu,Y_\mu)\indic_{\mu < \varepsilon} \\ + e^{-r\varepsilon}\int_0^1 \big( e^{-r (\nu-\varepsilon)} f(X_{\nu},Y_\nu) \indic_{\nu<\mu} +e^{-r (\mu-\varepsilon)} h(X_{\mu},Y_\mu) \indic_{\mu \leq\nu} \big)dv \indic_{\mu \geq \varepsilon}\Big]\Big).
\end{multline}
Let us fix $\nu \in \CT^Y_{m,\varepsilon}$ (or equivalently $\nu' \in \CT^Y_m$), $\mu \in \CT^X$ and $p\in \Delta(K)$.
By conditioning, we obtain:
\begin{align*}
\EE_{p,q}&\Big[ e^{-r \mu}h(X_\mu,Y_\mu)\indic_{\mu < \varepsilon} + e^{-r\varepsilon}\int_0^1\big( e^{-r (\nu-\varepsilon)} f(X_{\nu},Y_\nu) \indic_{\nu<\mu} +e^{-r (\mu-\varepsilon)} h(X_{\mu},Y_\mu) \indic_{\mu \leq\nu} \big)dv\indic_{\mu \geq \varepsilon}\Big] \\
& = \EE_{p,q}\Big[ e^{-r \mu}\EE_{p,q}\big[h(X_\mu,Y_\mu)|\CF^X_{\mu}\big]\indic_{\mu < \varepsilon}  \\
&\qquad + e^{-r\varepsilon} \EE_{p,q}\big[\int_0^1\big(e^{-r (\nu-\varepsilon)} f(X_{\nu},Y_\nu) \indic_{\nu<\mu} +e^{-r (\mu-\varepsilon)} h(X_{\mu},Y_\mu) \indic_{\mu \leq\nu} \big)dv |\CF^{X,Y}_{\varepsilon} \big]\indic_{\mu \geq \varepsilon}\Big] 
\end{align*}
Recall that the stopping time $\mu$ of the filtration $\CF^{X}$ can be identified with a stopping time defined on $\Omega_X$. It is well-known that on the event $\mu \geq \varepsilon$, we have $\mu=\mu'(\omega_X,\theta^X_{\varepsilon}(\omega_X))+\varepsilon$, where $\mu'$ is $\CF^X_{\varepsilon}\otimes \CF^X_{\infty}$ measurable and for all $\omega$, $\mu'(\omega,.)$ is an $\CF^{X}$ stopping time (see theorem 103 p. 151 in \cite{dellacheriemeyer}). Then, using the Markov property, we deduce that
\begin{align*}
\indic_{\mu \geq \varepsilon}\EE_{p,q}&[ \int_0^1 \left(e^{-r (\nu-\varepsilon)} f(X_{\nu},Y_\nu) \indic_{\nu<\mu} +e^{-r (\mu-\varepsilon)} h(X_{\mu},Y_\mu) \indic_{\mu \leq\nu}\right) dv |\CF^{X,Y}_{\varepsilon} ] \\
&=\indic_{\mu \geq \varepsilon}\EE_{\delta_{X_{\varepsilon}},\delta_{Y_{\varepsilon}}}[\bar{J}(\mu'(\omega,.),\nu')].
\end{align*}
On the other hand, since $X$ and $Y$ are independent we have 
\begin{equation}\label{uncorrelated} \EE_{p,q}[ \delta_{X_{t},Y_{t}} | \CF^X_{t} ]= \delta_{X_{t}}\otimes q_t \in \Delta(K\times L).
\end{equation}
Using the usual properties of the optional projection (see e.g. \cite{dellacheriemeyer})
the previous equality implies $\EE_{p,q}[h(X_\mu,Y_\mu) |\CF^X_{\mu}] =h(\delta_{X_{\mu}},q_{\mu})$, and  inequality \eqref{ineqDPP} may be rewritten as
\begin{align*}
V^{+,*}(x,q) \geq  \sup_{\nu' \in \CT^Y_m}\; &\inf_{\mu \in \CT^X}\; \inf_{p \in \Delta(K)}\\
& \Big(\langle x,p\rangle - \EE_{p,q}[ e^{-r \mu}h(\delta_{X_{\mu}},q_\mu)\indic_{\mu < \varepsilon} + e^{-r\varepsilon} \EE_{X_{\varepsilon},Y_{\varepsilon}}[\bar{J}(\mu'(\omega,.),\nu')]\indic_{\mu \geq \varepsilon}]\Big) .
\end{align*}
Defining
\[ \forall (\nu',p,q)\in \CT^Y_m\times \Delta(K)\times \Delta(L), \;F(\nu',p,q)= \sup_{\hat{\mu} \in \CT^X} \EE_{p,q}[\bar{J}(\hat{\mu},\nu')] \]
we have, using the same arguments as for \eqref{auseinander} :
\[F(\nu',p,q)=\sum_{\ell \in L}\sum_{k\in K} q^\ell p^k F(\nu',\delta_k,\delta_{\ell}).\]  
The previous inequality implies therefore
\begin{align}\label{ineqDPP1}
V^{+,*}(x,q) \geq \sup_{\nu' \in \CT^Y_m}\; &\inf_{\mu \in \CT^X}\; \inf_{p \in \Delta(K)}\nonumber\\
&    \Big(\langle x,p\rangle- \EE_{p,q}[ e^{-r \mu}h(\delta_{X_{\mu}},q_{\mu})\indic_{\mu < \varepsilon} + e^{-r\varepsilon} F(\nu',\delta_{X_{\varepsilon}}, \delta_{Y_{\varepsilon}})\indic_{\mu \geq \varepsilon}]\Big)
\end{align}
and taking conditional expectation with respect to $\CF^X_{\varepsilon}$, we obtain  
\begin{align*}
\EE_{p,q}[F(\nu', \delta_{X_{\varepsilon}}, \delta_{Y_{\varepsilon}})\indic_{\mu \geq \varepsilon} | \CF^X_{\varepsilon} ]=F(\nu', \delta_{X_{\varepsilon}}, q_\varepsilon)\indic_{\mu \geq \varepsilon}.
\end{align*}
Next, we apply the optional sampling theorem with the $\CF^X$-stopping time $\mu \wedge \varepsilon$ and obtain 
\begin{align*}\langle x,p\rangle &= \EE_{p,q}[ \langle x, e^{-\mu \tr R}\delta_{X_{\mu}} \rangle \indic_{\mu < \varepsilon} + \langle x, e^{-\varepsilon \tr R}\delta_{X_{{\varepsilon}}} \rangle \indic_{\mu \geq \varepsilon} ]\\
&= \EE_{p,q}[ \langle e^{-\mu R}x, \delta_{X_{\mu}}\rangle \indic_{\mu < \varepsilon} + \langle e^{-\varepsilon R}x,\delta_{X_{{\varepsilon}}}  \rangle \indic_{\mu \geq \varepsilon} ]. 
\end{align*}
Substituting the last two equalities in the right-hand side of \eqref{ineqDPP1} yields
\begin{align*}
\langle& x, p\rangle - \EE_{p,q}[ e^{-r \mu}h(\delta_{X_{\mu}},q_\mu)\indic_{\mu < \varepsilon} + e^{-r\varepsilon} F(\nu',\delta_{X_{{\varepsilon}}} , q_\varepsilon)\indic_{\mu \geq \varepsilon}]\\ 
 &= \EE_ {p,q}[ e^{-r \mu} (\langle e^{\mu(rI-R)}x, \delta_{X_{\mu}} \rangle - h(\delta_{X_{\mu}},q_\mu) )\indic_{\mu< \varepsilon} + e^{-r \varepsilon} ( \langle e^{\varepsilon(rI-R)}x, \delta_{X_{{\varepsilon}}} \rangle - F(\nu',\delta_{X_{{\varepsilon}}} , q_\varepsilon) ) \indic_{\mu \geq \varepsilon}]\\
&= \EE_ {p,q}[ e^{-r \mu} (\langle x_\mu, \delta_{X_{\mu}} \rangle - h(\delta_{X_{\mu}},q_\mu) )\indic_{\mu< \varepsilon} + e^{-r \varepsilon} ( \langle x_\varepsilon, \delta_{X_{{\varepsilon}}}\rangle - F(\nu',\delta_{X_{{\varepsilon}}} , q_\varepsilon) ) \indic_{\mu \geq \varepsilon}].
 \end{align*} 
Given any $\eta>0$, let us choose $\nu'$ as an $\eta$-optimal minimizer in the problem $V^{+,*}(x_{\varepsilon},q_\varepsilon)$ (note that these dynamics do not depend on $p$ or $\mu$), so that
\begin{align*}
\left( \langle x_\varepsilon, \delta_{X_{{\varepsilon}}} \rangle - F(\nu',\delta_{X_{{\varepsilon}}}, q_\varepsilon) \right) & \geq  \inf_{p' \in \Delta(K)}\left( \langle x_\varepsilon, p'\rangle - F(\nu',p', q_\varepsilon) \right) \\
&= \inf_{p' \in \Delta(K)}\inf_{\hat{\mu} \in \CT^X}\left( \langle x_\varepsilon, p' \rangle - \EE_{p',q_\varepsilon}[ \bar{J}(\hat{\mu},\nu')] \right) \\
&\geq   V^{+,*}(x_\varepsilon,q_\varepsilon)- \eta.
\end{align*}
Using the preceding results in (\ref{ineqDPP1}), we deduce that for all $\eta>0$
\begin{align*}
V^{+,*}&(x,q)\\
 &\geq \inf_{\mu \in \CT^X}\;\inf_{p \in \Delta(K)}\; \EE_{p,q}[ e^{-r \mu} (\langle x_\mu, \delta_{X_{\mu}}\rangle - h(\delta_{X_{\mu}},q_\mu) )\indic_{\mu< \varepsilon} + e^{-r \varepsilon}(V^{+,*}(x_\varepsilon,q_\varepsilon)-\eta) \indic_{\mu \geq \varepsilon}] \\
&\geq \inf_{\mu \in \CT^X}\;\inf_{p \in \Delta(K)}\; \EE_{p,q}[ e^{-r \mu} h^*(x_{\mu},q_\mu) \indic_{\mu< \varepsilon} + e^{-r \varepsilon}(V^{+,*}(x_\varepsilon,q_\varepsilon)-\eta) \indic_{\mu \geq \varepsilon}].
 \end{align*} 
Note that in the preceding expression, only $\mu$ is random, and thus we may replace the expectation with an integral with respect to the law of $\mu$ on $\overline{\RR}_+$ denoted $P_{p,\mu}$ which yields
\begin{align*}
V^{+,*}(x,q) 
&\geq \inf_{\mu \in \CT^X}\;\inf_{p \in \Delta(K)}\; \int_{\overline{\RR}_+} \Big(e^{-r t} h^*(x_{t},q_t) )\indic_{t< \varepsilon} + e^{-r \varepsilon}(V^{+,*}(x_\varepsilon,q_\varepsilon)-\eta) \indic_{t \geq \varepsilon}\Big) dP_{p,\mu}(t).
 \end{align*} 
Using the linearity of the integral and arguing as in Remark \ref{purifystoppingtimes}, the infimum over all admissible distributions for $\mu$ is equal to the infimum over Dirac masses (constant stopping times), and the conclusion follows by sending $\eta$ to zero.
\end{proof}

\subsection{Subsolution property for $V^+$}
We will prove the subsolution property for $V^+$ by establishing a super-solution property for $V^{+,*}$. The results rely on classical tools of convex analysis. 
\begin{notation}\label{differential}
Let $C \subset \RR^K$ and $D\subset \RR^L$ denote two convex sets. For any $g:C\times D \rightarrow \RR$ and $(x,y)\in C\times D$, we denote the sub-differential of $g$ with respect to the first variable $x$ by
\[\partial^-_1 g(x,y)=\{x^*\in \RR^K \,|\, \forall x'\in C , \; g(x,y)+\langle x^*, x'-x\rangle \leq g(x',y)\}.
\]
The super-differential $\partial^+_1g(x,y)$ is defined similarly.\\
We use the index $\partial_2^-$ (resp.  $\partial_2^+$), whenever we consider derivatives with respect to the second variable $y\in D$. We use $\partial^+$, $\partial^-$ for the full super- and sub-differential.
\end{notation}

\begin{proposition}\label{V*supersolbothsides2}
For all $(x,q)\in \RR^K\times \Delta(L)$, we have
\be\label{superdual}
\min\{(h^*-V^{+,*})(x,q) \,;\,  \vec{D}V^{+,*}(x,q ; (rI-R)x, \tr Q q)  - r V^{+*}(x,q) \}  \leq 0.
 \ee
\end{proposition}
\begin{proof} 
Recall that by Proposition \ref{DPPbothsides}, for all  $\varepsilon>0$
\begin{align}\label{DDPuse2}
V^{+,*}(x,q) &\geq \inf_{t \in[0,\varepsilon]}  \left(e^{-r t} h^*(x_t,q_t )\indic_{t< \varepsilon} + e^{-r \varepsilon}V^{+,*}(x_\varepsilon,q_\varepsilon) \indic_{t= \varepsilon}\right),
 \end{align}
where the dynamic $(x_t,q_t)$ is given by $x_t=x+\int_0^t  (rI-R)x_s ds$ and $q_t=q+\int_0^t {\tr Q } q_s ds$.
We know that $h^*(x,q)- V^{+,*}(x,q)\geq 0$ by construction. In case $h^*(x,q)- V^{+,*}(x,q)> 0$, there exists by continuity an $\tilde\varepsilon>0$ such that for all $0< \varepsilon \leq \tilde{\varepsilon}$, choosing  $t <\varepsilon$ would not be optimal in \eqref{DDPuse2}. Thus
\begin{align}
V^{+,*}(x,q) &\geq  e^{-r \varepsilon}V^{+,*}(x_{\varepsilon},q_{\varepsilon}).
\end{align}
We deduce that \eqref{superdual} holds since:
\[\vec{D}V^{+,*}(x,q ; (rI-R)x, \tr Q q)  - r V^{+,*}(x,q) =\lim_{\varepsilon \rightarrow 0} \frac{1}{\varepsilon}\left( e^{-r \varepsilon}V^{+,*}(x_{\varepsilon},q_{\varepsilon})-V^{+,*}(x,q)\right) \leq 0.\]
The last equality follows from the fact that $V^{+,*}$ is Lipschitz, implying that:
 \[ V^{+,*}(x_{\varepsilon},q_{\varepsilon})-V^{+,*}(x+ \varepsilon (rI-R)x,q+\varepsilon\tr Q q )= o(\varepsilon).\]
\end{proof}

\begin{proposition}\label{Vsubsolbothsides2}
$V^+$ is a subsolution of  \eqref{subpoint}. 
\end{proposition}

\begin{proof}
Again it is sufficient to prove only the subsolution property for $V^+$, as the proof of the supersolution property for $V^-$ is similar due to the symmetry of the problem. Since by construction $V^+(\bar{p},\bar{q}) \leq f(\bar{p},\bar{q})$ it remains to show that for $\bar{q}\in \Delta(L)$ and $\bar{p} \in Ext(V(.,\bar{q}))$  \[V^+(\bar{p},\bar{q})> h(\bar{p},\bar{q})\] implies
\be\label{ineqsubpoint} 
rV^+(\bar{p},\bar{q}) -\vec{D}_1V^+(\bar{p},\bar{q}; \tr R\bar{p}) - \vec{D}_2V^+(\bar{p},\bar{q} ; \tr Q \bar{q} ) \leq 0.
\ee 
We first assume that $\bar{p}$ is an exposed point of $V(.,\bar{q})$ (see Definition \ref{exposeddef}) and $V^{+}(\bar{p},\bar{q})>h(\bar{p},\bar{q})$.
We will reformulate \eqref{ineqsubpoint} using the conjugate function $ V^{+,*} $.
Let us choose $\bar{x} \in \partial^+_1V^+(\bar{p},\bar{q})$ and $\bar{y} \in \partial^-_2V^+(\bar{p},\bar{q})$ (see  Lemma \ref{superdiff}), such that
\[ \vec{D}_1V^+(\bar{p},\bar{q}; \tr R\bar{p})= \langle \bar{x}, \tr R \bar{p} \rangle, \;\;  \vec{D}_2V^+(\bar{p},\bar{q}; \tr Q\bar{q})= \langle \bar{y}, \tr Q \bar{q} \rangle.\]
By construction $V^+(\bar{p},\bar{q})=\langle \bar{x},\bar{p}  \rangle-V^{+,*}(\bar{x},\bar{q})$ and \eqref{ineqsubpoint} can be written as
\be \label{ineqsubpoint2}  
\langle (rI-R)\bar{x}, \bar{p} \rangle -\langle \tr Q \bar{q}, \bar{y} \rangle - r V^{+,*}(\bar{x},\bar{q}) \leq 0.
\ee
As $\bar{p}$ is an exposed point, we know that there exists some $\hat{x} \in \partial^+_1V^+(\bar{p},\bar{q})$ such that in the expression 
\[ V^{+,*}(\hat x,\bar q)= \inf_{p \in \Delta(K)} \langle \hat{x}, p \rangle - V^+(p,\bar{q}), \] 
the minimum is uniquely attained in $\bar{p}$. It follows that denoting $u:=\hat{x}-\bar{x}$, for all $\varepsilon>0$, the minimum in the expression 
\[  V^{+,*}(\bar{x}+\varepsilon u,\bar{q})=\inf_{p \in \Delta(K)} \langle \bar{x}+ \varepsilon u, p \rangle - V^+(p,\bar{q}), \] 
is uniquely attained in $\bar{p}$. Note that it may be that $\bar{x}=\hat{x}$ in which case $u=0$. 
Fenchel's lemma implies that the function $V^{+,*}(.,\bar{q})$ is differentiable at $\bar{x}+\varepsilon u$ with a gradient equal to $\bar{p}$ and we have 
\[ V^{+,*}(\bar{x}+\varepsilon u,\bar{q})= \langle \bar{x}+ \varepsilon u,\bar{p} \rangle - V^+(\bar{p},\bar{q}). \]
Instead of proving directly \eqref{ineqsubpoint2}, we will prove that for all $\varepsilon>0$
\be \label{ineqsubpoint3}  
\langle (rI-R)\bar{x}+\varepsilon u, \bar{p} \rangle -\langle \tr Q \bar{q}, \bar{y} \rangle - r V^{+,*}(\bar{x}+\varepsilon u,\bar{q}) \leq 0,
\ee
 \eqref{ineqsubpoint2} follows then by sending  $\varepsilon>0$ to zero.\p
In order to apply Proposition \ref{V*supersolbothsides2}, let us prove that 
$V^{+,*}(\bar{x}+\varepsilon u,\bar{q}) < h^*(\bar{x}+ \varepsilon u,\bar{q})$.
By construction, we have $V^{+,*}\leq h^*$ since $V^+ \geq h$. Assume by contradiction that $V^{+,*}(\bar{x}+\varepsilon u,\bar{q}) = h^*(\bar{x}+\varepsilon u,\bar{q})$. Both functions being concave with respect to their first argument and since $D_1V^{+,*}(\bar{x}+\varepsilon u,\bar{q})=\bar{p}$, we would have 
\[ D_1 h^*(\bar{x}+\varepsilon u,\bar{q})=D_1 V^{+,*}(\bar{x}+\varepsilon u,\bar{q})=\bar{p},\]
and therefore 
\[ V^+(\bar{p},\bar{q})=\langle \bar{x}+\varepsilon u,\bar{p} \rangle - V^{+,*}(\bar{x}+\varepsilon u,\bar{q}) =\langle \bar{x}+\varepsilon u,\bar{p} \rangle - h^*(\bar{x}+\varepsilon u,\bar{q})= h(\bar{p},\bar{q}),\]
which contradicts the assumption $V^+(\bar{p},\bar{q})-h(\bar{p},\bar{q})>0$.
\p
Note that $q \rightarrow V^+(\bar{p},q)$ is convex on $\Delta(L)$ and that $\bar{y}$ was chosen so that the directional derivative verifies (see Lemma \ref{superdiff} for the first equality) 
\[ \vec{D}_2 V^+(\bar{p},\bar{q} ; \tr Q\bar{q})=\max_{v \in \partial^-_2 V^+(\bar{p},\bar{q})} \langle v , \tr Q\bar{q} \rangle = \langle \bar{y} , \tr Q\bar{q} \rangle.\]
Since $V^{+,*}$ is a concave Lipschitz function on $\RR^K\times \Delta(L)$, the envelope theorem (see Lemma \ref{danskinlemma}) implies
\be \label{danskin2} \partial^+ V^{+,*}(\bar{x}+\varepsilon u,\bar{q})= \{ \bar{p}\} \times(-\partial^-_2V^+(\bar{p},\bar{q})).\ee
Indeed, note that the right-hand side of \eqref{danskin2} is the superdifferential of the concave function
\[ (x,q) \in \RR^K \times \Delta(L) \rightarrow \langle x, \bar{p} \rangle - V^+(\bar{p},q) ,\]
at $(\bar{x}+\varepsilon u, \bar{q})$. 
We deduce that the directional derivatives of $V^{+,*}$ at $(\bar{x}+\varepsilon u,\bar{q})$ verify
\begin{align*} 
\vec{D}V^{+,*}(\bar{x}+\varepsilon u,\bar{q}; (rI-R)\bar{x}, \tr Q \bar{q})& = \min_{(w,v) \in \partial^+ V^{+,*}(\bar{x}+\varepsilon u,\bar{q})} \langle w,(rI-R)\bar{x} \rangle + \langle v, \tr Q \bar{q} \rangle 
\\ &= \langle \bar{p},(rI-R)\bar{x} \rangle - \langle \bar{y}, \tr Q \bar{q} \rangle.
\end{align*} 
(\ref{ineqsubpoint3}) follows then from Proposition \ref{V*supersolbothsides2}.
\p
It remains to extend the result from exposed points to extreme points. We know that exposed points are a dense subset of extreme points (see Theorem 18.6 in \cite{rockafellar}). We may therefore use an approximation argument by using Lemma \ref{exposed}  applied to $p \rightarrow -V(., \bar{q})$ together with the fact that $p\rightarrow \vec{D}_2V^+(p,\bar{q})$ is upper semi-continuous (use Lemma \ref{superdiff} together with the fact that the correspondence of superdifferentials has closed graph). 
\end{proof}

\subsection{Symmetric properties for $V^-$}
{Let us briefly mention the properties satisfied by $V^-$ and its convex conjugate, as they will be used in the next section.  
Define the convex conjugate in $q$ of $V^{-}$ as 
\[\forall p\in \Delta(K), y \in \RR^L, \;  V^{-}_*(p,y):= \sup_{q\in \Delta(L)} \{\langle q,y \rangle - V^-(p,q)\}.\]
As in Remark \ref{purifystoppingtimes}, we have $h(p,q) \leq V^-(p,q) \leq f(p,q)$, and it follows that:
\be
\begin{array}{l}
f_*(p,y) \leq V^-_*(p,y) \leq h_*(p,y),\\
\end{array}
\ee
where the functions $h_*, f_*$ are defined as
\be
\begin{array}{l}
h_*(p,y):=\sup_{q \in \Delta(L)} \{\langle q,y \rangle -h(p,q)\},\;  f_*( p, y):=\sup_{q \in \Delta(L)} \{\langle q,y \rangle -f(p,q)\}.
\end{array}
\ee 
As in Proposition \ref{V*supersolbothsides2}, for all $(p,y)\in  \Delta(K)\times \RR^L$, we have
\be\label{subdual}
\max\{(f_*-V^{-}_*)(x,q) \,;\,  \vec{D}V^{-}_*(p,y ; \tr R p,(rI-Q)y)  - r V^{-}_*(p,y) \}  \geq 0.
\ee
As in Proposition \ref{Vsubsolbothsides2}, \eqref{subdual} implies that $V^-$ is a supersolution of  \eqref{superpoint}.}

\subsection{Comparison principle}\label{sectioncomparison}

In the previous section we showed that $V^+(p,q)$ is a subsolution to \eqref{subpoint} while $V^-(p,q)$ verifies the supersolution property \eqref{superpoint}. Since   $V^+(p,q)\geq V^-(p,q)$ by construction the following comparison principle will imply Theorem \ref{thmextremal}.
\p
Let us recall the classical definition of an extreme point for a convex set.
\begin{definition}\label{extremepoint}
Let $C \subset \RR^n$ a convex set. $x \in C$ is an extreme point of $C$ if for any $x_1,x_2 \in C$ and $\lambda \in [0,1]$:
\[ \lambda x_1 + (1-\lambda)x_2=x \Rightarrow x_1=x_2=x.\]
\end{definition}

\begin{theorem}\label{comparisonthm2}
Let $w_1,w_2$ be two Lipschitz continuous saddle functions defined on $\Delta(K)\times \Delta(L)$ such that $w_1$ verifies the subsolution property \eqref{subpoint} and $w_2$  verifies the supersolution property \eqref{superpoint}. Then $w_1 \leq w_2$.
\end{theorem}
\begin{proof}
We proceed by contradiction. Assume that $M:=\max_{\Delta(K)\times \Delta(L)} w_1-w_2 >0$ and let $C$ denote the compact set of $(p,q)\in \Delta(K)\times \Delta(L)$  where the maximum is reached. Let $(\bar{p},\bar{q})\in C$ denote an extreme point of the convex hull of $C$, which exists by Krein-Milman theorem and belongs to $C$ by definition of the convex hull. It follows that $\bar{p}$ is an extreme point of $w_1(\cdot,\bar{q})$ and that $\bar{q}$ is an extreme point of $w_1(\bar{p},.)$.

Let us prove this property for $\bar{p}$ (the case of $\bar{q}$ being symmetric).   
Assume that there exists $p_1,p_2\in \Delta(K)$ and $\lambda\in (0,1)$ such that $\bar{p}=\lambda p_1 +(1-\lambda)p_2$ and $\lambda w_1(p_1,\bar{q})+(1-\lambda)w_1(p_2,\bar{q})=w_1(\bar{p},\bar{q})$. Using that $w_2(.,\bar{q})$ is concave, we would have
\[\lambda (w_1(p_1,\bar{q})- w_2(p_1,\bar{q}))+(1-\lambda)(w_1(p_2,\bar{q})-w_2(p_2,\bar{q})) \geq w_1(\bar{p},\bar{q}) -w_2(\bar{p},\bar{q})=M. \]
As $w_1-w_2 \leq M$, we deduce that $(p_1,\bar{q})$ and $(p_2,\bar{q})$ belong to $C$ and therefore that $p_1=p_2$.
\p 
At point $(\bar{p},\bar{q})$, we have $f(\bar{p},\bar{q}) \geq w_1(\bar{p},\bar{q})> w_2(\bar{p},\bar{q})\geq h(\bar{p},\bar{q})$ so that 
\[ rw_1(\bar{p},\bar{q}) -\vec{D}_1 w_1(\bar{p},\bar{q}; \tr R\bar{p}) - \vec{D}_2 w_1(\bar{p},\bar{q} ; \tr Q \bar{q} ) \leq 0,\]
\[ rw_2(\bar{p},\bar{q}) -\vec{D}_1 w_2(\bar{p},\bar{q}; \tr R\bar{p}) - \vec{D}_2 w_2(\bar{p},\bar{q} ; \tr Q \bar{q} ) \geq 0.\]
Note that $\tr Rp$ (resp. $\tr Qq$) always belong to the tangent cone of $\Delta(K)$ at $p$ (resp. of $\Delta(L)$ at $q$) so that directional derivatives are well-defined and real-valued. We deduce that 
\begin{align*}
\vec{D}_1 w_1(\bar{p},\bar{q}; \tr R\bar{p}) + \vec{D}_2 w_1(\bar{p},\bar{q} ; \tr Q \bar{q} ) &\geq rw_1(\bar{p},\bar{q})  \geq  rw_2(\bar{p},\bar{q})+rM \\ &\geq \vec{D}_1 w_2(\bar{p},\bar{q}; \tr R\bar{p}) + \vec{D}_2 w_2(\bar{p},\bar{q} ; \tr Q \bar{q}) +rM.
\end{align*}
It follows that one of the following inequalities holds true:
\[ \vec{D}_1 w_1(\bar{p},\bar{q}; \tr R\bar{p})>\vec{D}_1 w_2(\bar{p},\bar{q}; \tr R\bar{p})\] 
\[ \vec{D}_2 w_1(\bar{p},\bar{q} ; \tr Q \bar{q} ) > \vec{D}_2 w_2(\bar{p},\bar{q} ; \tr Q \bar{q}).\]
In the first case, this would imply that for a sufficiently small $\varepsilon$,
\[ w_1(\bar{p}+\varepsilon \tr R \bar{p},\bar{q}) - w_2(\bar{p}+\varepsilon \tr R \bar{p},\bar{q}) > w_1(\bar{p},\bar{q}) - w_2(\bar{p},\bar{q})=M,\]
and thus a contradiction. The second case is similar and this concludes the proof.
\end{proof}

\section{Optimal stopping times}\label{sectionverif}

Using the PDE characterization of the value function $V:\Delta(K)\times\Delta(L)\rightarrow\mathbb{R}$ by Theorem \ref{thmextremal}, it is possible to give a verification theorem for mixed stopping times for both players. Here we present the characterization of optimal stopping times $\mu$ for player 1, i.e. the player observing $X$. By symmetry of the problem the result for player 2 is given in a similar way.\\

{Our construction being quite technical, we provide below a heuristic description of our result which emphasizes the relations with the existing literature on dynamic games with incomplete information.} In order to characterize optimal stopping times $\mu$ for player 1, we introduce the belief process of the uninformed player over $X$.  Let us fix a mixed stopping time $\mu \in \CT^X_m$. Despite the fact that player 2 has no information on $X$, he can compute for any $t\geq0$ the conditional distribution $\pi_t$ (defined below) of $X_t$ given the event that player 1 did not stop before time $t$. To that end we consider the product probability space 
\[(\Omega',\CF',\PP'_{p,q}):=(\Omega_X \times \Omega_Y \times [0,1], \CF^X_{\infty}\otimes \CF^Y_\infty \otimes \CB([0,1]), \PP_{p,q}\otimes \Leb),\]
where $\Leb$ stands for the Lebesgue measure. The stopping time $\mu(\omega_X,u)$ is thus seen as a random variable defined on $\Omega'$.\\
We define the belief process $\pi$ taking values in $\Delta(K)$ as a c\`{a}dl\`{a}g version of:
\be\label{pi}
 (\PP'_{p,q} [ X_t=k | \CH^{\mu}_t] )_{k\in K},\; t\geq 0,
 \ee
where $ \CH^{\mu}$ is the usual right-continuous augmentation of  $\sigma(\indic_{\mu \leq s}, 0\leq s \leq t) $.
By construction, the process $\pi$ has the following property: 
\be\label{martprop}
\forall \; 0 \leq s \leq t, \; \EE_{\PP'_{p,q}}[ \pi_t | \CF^\pi_s]= e^{(t-s)\tr R}\pi_s.
\ee 

{Let us call a belief process optimal if there exists an optimal stopping time of player $1$ inducing this belief process. It is classical in the literature on zero-sum dynamic games with incomplete information to search for conditions ensuring that a belief process is optimal. This idea appears explicitly in De Meyer \cite{demeyergeb}, Cardaliaguet-Rainer \cite{cardaexemple}, Gr\"un \cite{grunstopping,grungirsanov} and more recently in Cardaliaguet {\it et al.} \cite{cardaetal}. Our main result has the same flavor as Theorem 4.11 in \cite{cardaexemple} where a set of sufficient conditions is given for a belief process to be optimal. Note however that all these works were dealing with games with incomplete information on one side. In order to adapt these techniques to our model with incomplete information on both sides, we introduce a virtual dual process $\xi$ taking values in $\RR^L$, and consider the extended belief process $Z=(\pi,\xi)$. This idea was already present in the work of Heuer  \cite{heuer}, in which (adapting the notation to that of our model) the strategy of player 1 was based on a virtual belief process $q_t$ about the unknown state $Y_t$ of player $2$ and on the corresponding dual variable $\xi_t \in \partial^-_2 V_*(\pi_t,q_t)$. We work here directly with the dual process $\xi$ which is controlled by player $1$. Precisely, the process $\xi$ is c\`{a}dl\`{a}g, $\CF^\pi$-measurable and has to satisfy the property
\be\label{martprop2}
\forall \; 0 \leq s \leq t, \; \EE_{\PP'_{p,q}}[ \xi_t | \CF^\pi_s]= e^{(t-s)(rI-Q)}\xi_s,
\ee
which is related to the variational inequality \eqref{subdual} satisfied by $V_*$. The process $Z$ is therefore a pure jump process.
We now consider the following two closed subsets of $\Delta(K)\times \RR^L$ 
\[\mathcal{H}:=\{(p,y)\in\Delta(K)\times \RR^L \,|\, rV_*(p,y) - \vec{D}V_*(p,y; \tr R p, (rI-Q)y) \geq 0 \},\]
\[ \mathcal{S}:=\{(p,y)\in\Delta(K)\times \RR^L \,|\, V_*(p,y)=h_*(p,y)\}.\]
The set $\mathcal{H}$ is related to the set called non-revealing set in \cite{cardaexemple}. 
Let us first describe informally how player $1$ can control the process $Z$, which basically reduces to three possible behaviors:
\begin{itemize}
\item Player $1$ does not stop during a time interval $[t,t+dt]$, and  thanks to the conditions \eqref{martprop}, \eqref{martprop2}, the process $Z$ follows the deterministic dynamic $dZ_t=AZ_tdt$ where 
\[A=\begin{pmatrix}  \tr R & 0 \\ 0 & rI-Q \end{pmatrix}.\] 
\item Player $1$ stops with positive probability at some time $t$, which induces a jump of the belief process $\pi$, centered in $\pi_{t-}$, and such that $\pi_t$ takes at most two different values corresponding respectively to the events that player $1$ did stop and did not stop at time $t$. The process $\xi$ may also have a jump at time $t$, centered in $\xi_{t-}$ and having at most two endpoints corresponding to the same events, however the endpoints of this jump can be chosen arbitrarily. 
\item If player $1$ stops with positive intensity $\rho$ (which may depend on $X_t$) during some time interval $[t,t+dt]$, then conditionally on the fact that player $1$ did not stop, the process $\pi$ follows a deterministic continuous trajectory, and jumps only when player $1$ stops. The behavior of $\xi$ is similar, but the deterministic trajectory as well as the jump are not determined by the stopping strategy of player $1$ and can be chosen arbitrarily among those satisfying \eqref{martprop2}.  
\end{itemize}
Taking into account the stationarity of the model, we reduce the analysis to Markovian processes $Z$ having the above-mentioned properties, which are called piecewise deterministic Markov process (PDMP) and have been introduced by Davis \cite{Davis}.
Precisely, we now focus on stopping times $\mu$ which induce a belief process $\pi$ that can be extended on the time interval $[0,\mu]$ to a PDMP $Z=(\pi,\xi)$ in the family formally described below.}
{An informal statement of our main theorem is the following: $\mu$ is an optimal stopping time in the game with initial probabilities $(p,q)$ if the induced belief process $\pi$ can be extended to a process $Z=(\pi,\xi)$ which is a PDMP on the time interval $[0,\mu]$, with state space $E \subset \CH \cup \CS$, random initial condition $Z_{0}$ and such that 
\begin{itemize}
\item $Z_{0-}:=\EE[Z_0]=(p,y)$ with $y \in \partial_2^- V_*(p,q)$,
\item $Z_t \in \CH$ on $\{t<\mu\}$ and $Z_\mu \in \CS$, 
\item jumps may occur only at times $0$ and $\mu$ and over the flat parts of $V_*$,
\item $V_*$ is regular enough.
\end{itemize}
The conditions that the belief process has to stay in the set $\CH$ and can only jump over the flat parts of the value function were already present in theorem 4.11 in \cite{cardaexemple}. Note however that we consider here the conjugate value $V_*$ so that that both the set $\CH$ and the definition of flat parts depend on the additional variable $\xi$. Another important difference is the regularity assumption: in \cite{cardaexemple}, the value function was assumed to be $C^2$ with respect to the state variable in order to apply It\^{o} formula, whereas we only assume here a very weak differentiability condition (see \eqref{conederivative} below). Note that such a weak condition is necessary in order to apply our theorem to the examples studied in section \ref{sectionexample} where $V_*$ is not globally $C^1$.} 
\p
Let us now describe precisely the type of PDMP we will consider (see Davis \cite{Davis} for more details).
Let $E$ denote a closed subset of $\Delta(K)\times \RR^L$ and let $(\alpha,\lambda,\phi)$ be given with:
\begin{itemize}
\item $\alpha: E \rightarrow \RR^K\times \RR^L$:  locally bounded measurable vector field such that for all initial point $z \in E$, there exists a unique global solution of $w_z'(t)=\alpha(w_z(t))$ with initial condition $w_z(0)=z$ which stays in $E$. 
\item $\lambda: E \rightarrow \RR_+$: locally bounded measurable intensity function. 
\item $\phi: E \rightarrow E$: locally bounded measurable jump function such that for all $z\in E$, $\phi(z)\neq z$ 
\item We assume in addition that for all $z\in E$, the maps $t \rightarrow \alpha(w_z(t)), \lambda(w_z(t)), \phi(w_z(t))$ are c\`{a}dl\`{a}g.
\end{itemize}
The construction of the PDMP $Z$ with characteristics $(\alpha,\lambda,\phi)$ and initial position $Z_0=z$ is as follows: 
Let $T_1$ denote a non-negative random variable such that $\PP(T_1 >t)= \exp(-\int_0^t \lambda(w_z(s) )ds)$.
Then, define $Z_t=w_z(t)$ for $t \in [0,T_1)$ and $Z_{T_1}=\phi(w_z(T_1))$. For $k\geq 2$, construct by induction the variables $T_k-T_{k-1}$, $Z_{T_k}$, and the process $Z$ on $[T_{k-1}, T_k]$ using the same method where  $z$ is replaced by $Z_{T_{k-1}}$.
\p
The process $Z$ has locally Lipschitz trajectories on the intervals $[T_{k-1}, T_k)$ and jumps at times $T_k$.
We also assume that for all $z\in E$ and all $t \geq 0$, $\EE[ \sum_{s \leq t} |Z_s- Z_{s-}| ]< \infty$. 
This condition ensures that the sequence $T_k$ goes to $+\infty$ so that the process $Z$ is well-defined for all times.
\p
We now introduce a set of assumptions, called structure conditions on $E$ and $(\alpha,\lambda,\phi)$ that will be used in our verification theorem below. 

\begin{definition} \label{structure} 
We say that $E$ and $(\alpha,\lambda,\phi)$ fulfill the structure conditions $(SC)$ if
\begin{itemize}
\item[(SC1)] $E= E_\CH \cup \CS$ with $E_\CH$ a nonempty closed subset of $\CH$, and $E_{\CH} \cap \CS=\emptyset$.
\item[(SC2)] For all $z\in \CS$, $\lambda(z)=0$ and $\alpha(z)=0$. For all $z\in E_\CH$, for all $t\geq 0$, $w_z(t)\in E_\CH$.
\item[(SC3)] $\phi: E_\CH \rightarrow \CS$ and for all $z\in E_\CH$, if we denote $\phi(z)=(\phi_p(z),\phi_y(z))$, then 
\[ \{ k \in K | p_k=0\} \subset \{ k \in K |\phi_p(z)_k=0 \}.\] 
\item[(SC4)] For all $z \in E_\CH$:
\be \label{flatjump}
\lambda(z)\Big(V_*(\phi(z))-V_*(z)-\vec{D}V_*(z;\phi(z)-z)\Big)=0 .
\ee
\be \label{conederivative}
\vec{D}V_*(z;\alpha(z))+\vec{D}V_*(z;\lambda(z)(\phi(z)-z)) =\vec{D}V_*\left(z;\alpha(z)+\lambda(z)(\phi(z)-z)\right).
\ee
\be \label{structuredynamic}
\alpha(z)+  \lambda(z)(\phi(z)-z)= A z \;\text{with}\; A=\begin{pmatrix}  \tr R & 0 \\ 0 & rI-Q \end{pmatrix}
\ee
\item[(SC5)] For all $z\in (\Delta(K)\times \RR^L) \setminus E$, there exists $z' \in E_\CH$, $z'' \in \CS$ and $m\in [0,1]$ such that  
\[ z= (1-m)z' +mz'', \; V_*(z)= (1-m)V_*(z')+m V_*(z'').\]
\end{itemize}
\end{definition}
\p
\textbf{Comments:}  
\p
\begin{itemize}
\item Condition $(SC2)$ means that all points in $\CS$ are absorbing and that $E_\CH$ is invariant by the flow associated to the ODE $w'=\alpha(w)$.
\item Condition $(SC3)$ means that the direction of jump is deterministic and that any jump should start in $E_\CH$ and end in $\CS$. $(SC2)$ and $(SC3)$ together ensure that the PDMP $Z$ with characteristics $(\alpha,\lambda,\phi)$ is well-defined and integrable for any starting point in $E$. 
\item Condition \eqref{flatjump} states that jumps occur only over the flat parts of the graph of $V_*$.
\item Since $V_*$ is convex, condition \eqref{conederivative} means that for all $z$, $V_*$ is differentiable at $z$ in the direction of  the cone generated by $\alpha(z)$ and the vector $\lambda(z)(\phi(z)-z)$. Note that this is automatically true when $\lambda(z)=0$ as $V_*$ admits directional derivatives. 
\item Condition \eqref{structuredynamic} was designed in such a way that if $Z$ is a PDMP with characteristics $(\alpha,\lambda,\phi)$ with $Z_0=z\in \CH$, then
\be \label{structurexi} \forall t \geq 0, \;\EE[Z_{t\wedge T_1}]= \EE[ \int_{0}^{t\wedge T_1} A Z_r dr ].\ee
Indeed, according to Theorem 5.5 in Davis \cite{Davis}, the identity map belongs to the extended generator of the process $Z$ so that Dynkin's formula applies, implying that:
\begin{align*}
\EE[Z_{t\wedge T_1} ] &= \EE[\int_{0}^{t\wedge T_1} \left(\alpha(Z_r)+ \lambda(Z_r)(\phi(Z_r)-Z_r) \right) dr ]\\
&=\EE[ \int_{0}^{t\wedge T_1} A Z_r dr ].
\end{align*}
\item Condition $(SC5)$ means that for any point $z$ outside of $E$, there exists a random initial condition $Z_0$ for the process, such that $Z_0$ takes only two values, one in $\CS$ and one in $E_\CH$, such that $\EE[Z_0]=z$ and $\EE[V_*(Z_0)]=V_*(z)$. With the convention $Z_{0-}:=z$, this last equality means that the process jumps at time $0$ over a flat part of $V_*$.
\end{itemize}


\p
We are now ready to state our verification Theorem.
\begin{theorem}\label{suffcientbothsides}
Let $(p,q) \in \Delta(K)\times \Delta(L)$ and $y \in \partial_2^- V(p,q)$. If there exists characteristics $(\alpha,\lambda,\phi)$ fulfilling conditions $(SC)$, then the stopping time $\mu \in \CT^X_m$ defined below is optimal. 
\p
\textbf{Case 1:} If $z=(p,y) \in E_\CH$, let us denote $z_t=(p_t,y_t)=w_z(t)$. \\
Recall that $(\omega,u)$ denotes the canonical element in $\Omega'$ and that $u$ is independent of $\omega$ and uniformly distributed over $[0,1]$. Let us consider a sequence of independent variables $(U_n(u))_{n \geq 0}$ uniformly distributed over $[0,1]$.
Given $t\geq 0$ and $k\in K$, define as a measurable function of $\tilde{u} \in  [0,1]$ the unique non-decreasing map  $M(t,k)(\tilde{u})$ with values in $[0,+\infty]$ such that:
\[ \int_0^1 \indic_{\{M(t,k)(\tilde{u}) \geq x\}} d\tilde{u} = \exp(-\int_t^{t+x} \frac{\phi_p(z_{s})_k}{(p_{s})_k} \lambda(z_{s})ds ),\]
with the convention $\frac{0}{0}=0$. 
Let $(S_n)_{n \geq 1}$ denote the sequence of jumps of the Markov chain $X$ and let $S_0=0$. 
Define, with the convention $\inf\emptyset=+\infty$:
\[ \mu_n(\omega,u)  = \inf \{t \in [S_n, S_{n+1}) \,|\,  t - S_n \geq M(S_n,X_{S_n})(U_n(u)) \}.\]
Then $\mu_z(\omega,u)=\inf_{n\geq 0} \mu_n(\omega,u)$ is optimal. 
\p
\textbf{Case 2:} If $z \notin E$, let $z'=(p',y')$, $z''=(p'',y'')$ and $m \in (0,1)$ be given by condition $(SC5)$. Then an optimal stopping time is given by: 
\[ \mu(\omega,u)= 0. \indic_{u \leq m, \; X_0 \in supp(p')} + \mu_{z'}(\omega, \frac{u-m}{1-m})\indic_{u>m},\]
where $\mu_{z'}$ is the stopping time constructed in case 1 and $supp(p')$ denotes the support of $p'$.
\p
\textbf{Case 3:} If $z \in \CS$, $\mu=0$ is optimal.
\end{theorem}


{
The construction of the characteristics $(\alpha,\lambda,\phi)$  is illustrated through two examples in section \ref{sectionexample} where a precise description of the optimal stopping time $\mu$ is given.   
The structure of the proof is as follows. In part 1 we prove that the belief process $\pi$ associated with the stopping time $\mu$, can be extended to a process $Z=(\pi,\xi)$ which is a PDMP with characteristics $(\alpha,\lambda,\phi)$. This first part is purely technical, as the definition of $\mu$ was designed exactly for this result to hold. Part 2 is more classical, we apply Dynkin's formula to $V_*(Z)$, and using the structure conditions on $Z$, this leads to the inequality \eqref{eq_dualopt}. Using then duality arguments similar to the one used in section \ref{sectionresult}, we prove that \eqref{eq_dualopt} implies the optimality of $\mu$.}

\begin{proof}
\p
\textbf{Part 1:} We construct a process $Z$ which is an $\CH^\mu$-PDMP with state space $E \subset \Delta(K)\times \RR^L$ and characteristics $(\alpha,\lambda,\phi)$. Some details are omitted as this first part relies on classical probabilistic arguments. Throughout the proof, $\PP$ stands for $\PP'_{p,q}$, and $\EE$ for the associated expectation operator.
\p
$\bullet$ We consider at first the case $z \in E_\CH$. 
\p
Define the belief process $\pi$ by (\ref{pi}). Let us denote $z_t=(p_t,y_t)=w_z(t)$.
Define then the process $\xi$ by 
\[ \forall t \geq 0, \;\xi_t= y_t \indic_{t<\mu} + \phi_y(z_\mu) \indic_{t \geq \mu}.\]  
Note that $\xi$ is an $\CH^\mu$-adapted c\`adl\`ag random process taking values in $\RR^L$.
\p
We will prove that the process $Z$ with $Z_t:=(\pi_{t \wedge \mu},\xi_t)$  is an $\CH^\mu$-PDMP with state space $E \subset \Delta(K)\times \RR^L$, characteristics $(\alpha,\lambda,\phi)$, having a unique jump at time $\mu$. 
\p
Let $T:=\sup\{ t \,|\, \PP( \mu >t)>0 \}$. By definition of $\mu$, we have  $T>0$ and for all $k\in K$ and all  $0 \leq t <T$ such that $(p_t)_k >0$:
\be \label{condintensity}
 \frac{1}{h}\left(\PP(\mu \in(t,t+h) | X_t=k, \mu >t)-\int_t^{t+h} \frac{\phi_p(z_s)_k}{(p_s)_k} \lambda(z_s)ds \right) \underset{h \rightarrow 0+}{\longrightarrow} 0,
\ee
and the convergence is locally uniform in $t$.
The proof  is omitted as it follows from standard properties of jump processes. 
\p
We claim that:
\be \label{edoforpi}
\forall t \geq 0, \; \pi_{t\wedge \mu} = p_t \indic_{t<\mu}+ \pi_{\mu} \indic_{t \geq \mu}.
\ee
Indeed, for all $t\in[0,T)$, on $\{t<\mu\}$, we have $\pi_t(k)=\PP(X_t=k | \mu>t )$ so that we only have to check that
\[ \PP(X_t=k | \mu>t ) =p_t(k).\]
Let $c(t)_k:=\PP(X_t=k | \mu>t)$, $b(t):=\PP(\mu>t)$ for $t\in[0,T)$. By conditioning on the events $\{X_t=k, \mu >t\}$ for all $k\in K$, we have for all $t$ in $[0,T)$:
\be \label{edob}
\frac{1}{h} (b(t+h)-b(t))  \underset{h \rightarrow 0+}{\longrightarrow} b'(t)=- b(t)\sum_k c(t)_k \frac{\phi_p(z_t)_k}{(p_t)_k} \lambda(z_t).
\ee

\p
On the other hand, $X$ is Markov with respect to the filtration $t \rightarrow \CF^X_t \vee \sigma(u)$ and $\{\mu>t\}$ belongs to the completion of $\CF^X_t \vee \sigma(u)$ (see e.g. Proposition 1 in \cite{tsirelson}), so that the conditional law of $(X_s)_{s \geq t}$ given $\{\mu>t\}$ is $\PP_{c(t)}$. We deduce that
\[\frac{1}{h} (\PP( X_{t+h}=k | \mu >t ) - \PP(X_t=k | \mu >t)) \underset{h \rightarrow 0+}{\longrightarrow} (\tr R c(t))_k  .\]
We have:
\begin{align*}
\frac{1}{h} (c(t+h)_k-c(t)_k)&= \frac{1}{h} [\PP( X_{t+h}=k | \mu >t+h ) - \PP(X_{t+h}=k | \mu >t)\\
&\qquad+\PP( X_{t+h}=k | \mu >t ) - \PP(X_t=k | \mu >t)] \\
&=  \frac{1}{h} \frac{1}{b(t)}(\PP( X_{t+h}=k , \mu >t+h ) - \PP(X_{t+h}=k , \mu >t)) \\
&\qquad + \frac{1}{h} \PP( X_{t+h}=k , \mu >t+h )(\frac{1}{b(t+h)}-\frac{1}{b(t)}) \\
&\qquad +\frac{1}{h} (\PP( X_{t+h}=k | \mu >t ) - \PP(X_t=k | \mu >t)
\end{align*}
Using that
\begin{align*}
\PP( X_{t+h}=k , \mu >t+h ) - \PP(X_{t+h}=k , \mu >t) &= -\PP(X_{t+h}=k, \mu \in (t,t+h]) \\
&=-\PP(X_t=k, \mu \in(t,t+h])+o(h),
\end{align*}
we deduce that for all $t\in[0,T)$:
\[\frac{1}{h} \frac{1}{b(t)}(\PP( X_{t+h}=k , \mu >t+h ) - \PP(X_{t+h}=k , \mu >t))  \underset{h \rightarrow 0+}{\longrightarrow} -c(t)_k \frac{\phi_p(z_t)_k}{(p_t)_k} \lambda(z_t).\]
We conclude that for all $t\in[0,T)$:
\[\frac{1}{h} (c(t+h)_k-c(t)_k)\underset{h \rightarrow 0+}{\longrightarrow} c'(t)_k=(\tr R c(t))_k -\lambda(z_t) \left[ \phi_p(z_t)_k \frac{c(t)_k}{(p_t)_k}  - c(t)_k\sum_{k'}  \frac{c(t)_{k'}}{(p_t)_{k'}}\phi_p(z_t)_{k'}\right].\]
Using condition \eqref{structuredynamic}, we can check that $t \rightarrow p_t$ is the unique solution of the above differential equation, so that $c(t)=p(t)$ for all $t \in [0,T)$ since $c(0)=p(0)$. This concludes the proof of \eqref{edoforpi}.
We deduce also that for all $t\geq 0$, we have: 
\[ b(t)=\PP(\mu>t)= \exp(-\int_0^t \lambda(z_s) ds).\]
Indeed, using that $p_t=c(t)$, the differential equation in the right-hand side of \eqref{edob} admits the above expression as unique solution.
\p
We now prove that $\pi_{\mu}=\phi_p(z_{\mu})$ on $\{\mu <\infty\}$. 
We start with the fact that
\[ \PP(X_t=k  | \mu \in(t,t+h])= -b(t)(p_t)_k\frac{\PP( \mu \in(t,t+h] | X_t=k, \mu >t)}{b(t+h)-b(t)}.\]
For all $t\in [0,T)$, we have therefore:
\be \label{limitlaw}  \PP(X_t=k  | \mu \in(t,t+h])\underset{h \rightarrow 0+}{\longrightarrow} \phi_p(z_t)_k ,
\ee
and the convergence is locally uniform in $t$. 
\p
For all $n\geq 1$ and all $i \in \NN$, let $t^n_i = \frac{i}{2^n}$. 
Let $\CH^n$ denote the $\sigma$-field generated by the events $(\{\mu \in (t^n_i,t^n_{i+1}]\})_{i \geq 0}$. Then $\sigma(\mu)=\bigvee_{n \geq 0} \CH^n$, and (up to null sets) $\CH^\mu_{\mu}=\sigma(\mu)$.
For all $k\in K$, we have by the martingale convergence theorem:
\[ (\pi_{\mu})_k\indic_{\mu <\infty} = \PP(X_{\mu}=k | \CH^\mu_\mu)\indic_{\mu <\infty}  = \lim_{n \rightarrow \infty} \PP( X_{\mu}=k | \CH^n )\indic_{\mu <\infty}\quad  a.s.\]
On the other hand, for all $n\geq 0$, we have:
\begin{align*}
\PP( X_{\mu}=k | \CH^n )\indic_{\mu <\infty} = \sum_{i \geq 0\,:\, t^n_i <T} \indic_{\mu \in (t^n_i,t^n_{i+1}]} \PP( X_{\mu}=k |\mu \in (t^n_i,t^n_{i+1}]).
\end{align*}
The following convergence holds locally uniformly in $t$ for $t\in[0,T)$:
\[ \left(  \PP( X_{\mu}=k |\mu \in (t,t+\frac{1}{2^n}])-\PP( X_{t}=k |\mu \in (t,t+\frac{1}{2^n}])\right) \underset{n \rightarrow \infty}{\longrightarrow} 0.\]
We deduce easily that the following convergence holds almost surely:
\[\PP( X_{\mu}=k | \CH^n )\indic_{\mu <\infty} \rightarrow \phi_p(z_\mu)_k \indic_{\mu <\infty},\]
which concludes the proof that $\pi_{\mu}=\phi_p(z_{\mu})$ on $\{\mu <\infty\}$. 
\p
To summarize, the process $Z$ is such that for all $t\geq 0$
\[ Z_t = z_t\indic_{t <\mu} + \phi(z_\mu)\indic_{t \geq \mu},\]
and $\mu$ is such that $\PP(\mu >t)=\exp(-\int_0^t \lambda(z_s) ds)$. As, up to null sets, $\CH^\mu$ is the natural filtration of $Z$, we deduce that $Z$ is an $\CH^\mu$-PDMP with characteristics $(\alpha,\lambda,\phi)$, having a unique jump at time $\mu$.
\p
$\bullet$ If $z\notin E$, define the belief process $\pi$ by (\ref{pi}). Then $\pi_0= z' \indic_{\mu >0} + z''\indic_{\mu=0}$ and $\PP(\mu>0)=1-m$. Let us denote $z_t=(p_t,y_t)=w_{z'}(t)$. Define then the process $\xi$ by 
\[ \forall t \geq 0, \;\xi_t= y_t \indic_{t<\mu} + \phi_y(z_\mu) \indic_{t \geq \mu}.\]  
As for the previous case, the process $Z$ with $Z_t:=(\pi_{t \wedge \mu},\xi_t)$  is an $\CH^\mu$-PDMP with state space $E \subset \Delta(K)\times \RR^L$, characteristics $(\alpha,\lambda,\phi)$, having a unique jump at time $\mu$. Indeed, reasoning conditionally on the event $\{\mu>0\}$, the same analysis as above shows that for all $t\geq 0$
\[ Z_t\indic_{\mu>0} = z_t\indic_{0<t <\mu} + \phi(z_\mu)\indic_{t \geq \mu>0},\]
and that $\mu$ is such that $\PP(\mu >t | \mu >0)=\exp(-\int_0^t \lambda(z_s) ds)$ for all $t\geq 0$. On the other hand, on the event $\{\mu=0\}$, we have $Z_t=Z_0=z'' \in \CS$ for all $t \geq 0$, which concludes the proof.
\p
$\bullet$ Case $z\in \CS$. The process defined by $Z_t=z$ for all $t\geq 0$, is an $\CH^\mu$-PDMP with state space $E \subset \Delta(K)\times \RR^L$, characteristics $(\alpha,\lambda,\phi)$.
\p
\textbf{Part 2:} We prove the optimality of $\mu$. Thanks to the results of part 1, we have constructed an $\CH^\mu$-PDMP $Z_t=(\pi_t,\xi_t)$ with state space $E \subset \Delta(K)\times \RR^L$ and characteristics $(\alpha,\lambda,\phi)$ such that
\begin{itemize}
\item[(i)] $Z_0$ takes finitely many values and:
\[\EE[Z_0]=(p,y), \;\EE[V_*(Z_0)]=V_*(p,y).\]
\item[(ii)] $Z_t \in\mathcal{H}$ on $[0,\mu)$, $Z$ may jump only at time $\mu$ and $Z_\mu \in \mathcal{S}$.
\end{itemize}
\p
As $V_*$ is convex and Lipschitz, and the vector field $\alpha$ is locally bounded, the map $t\rightarrow V_*(w_{\tilde{z}}(t))$ is locally Lipschitz for all $\tilde{z}\in E$, and therefore absolutely continuous with right-derivative $\vec{D}V_*(w_{\tilde{z}}(t);\alpha(w_{\tilde{z}}(t))$. It follows  that $V_*$ belongs to the extended generator of the PDMP $Z$ (see Theorem 5.5 in Davis \cite{Davis}) and Dynkin's formula gives for any  $t \in \overline{\RR}_+$
\begin{equation}\label{chainrule2}
\begin{array}{rcl}
V_*(p,y)&=&\EE[ V_*(Z_0)] \\
&=&\EE[ e^{-r(\mu\wedge t)} V_*(Z_{\mu\wedge t})+\int_0^{\mu\wedge t} r e^{-r s} V_*(\pi_s,\xi_s) ds\\
& &  -  \int_0^{\mu\wedge t}  e^{-r s} (\vec{D}V_* (Z_s; \alpha(Z_s))+  \lambda(Z_s)(V_*(\phi(z))-V_*(Z_s)) )ds ]
\end{array}
\end{equation}
\p
Using properties $(SC4)$, we deduce that:
\begin{align*}
V_*(p,y)&=  \EE[ e^{-r(\mu\wedge t)} V_*(Z_{\mu\wedge t})+\int_0^{\mu\wedge t}  e^{-r s} \left( rV_*(\pi_s,\xi_s) -\vec{D}V_* (Z_s; \alpha(Z_s))\right.  \\ &\qquad+ \left. \vec{D}V_* (Z_s;\lambda(Z_s)(\phi(Z_s)-Z_s))\right)ds ]\\
&=\EE[ e^{-r(\mu\wedge t)} V_*(Z_{\mu\wedge t}) \\& \qquad+\int_0^{\mu\wedge t}  e^{-r s}\Big(  rV_*(Z_s)  -  \vec{D}V_* \Big(Z_s; \alpha(Z_s)+  \lambda(Z_s)(\phi(Z_s)-Z_s)\Big) \Big)ds ]\\
&=\EE[ e^{-r(\mu\wedge t)} V_*(Z_{\mu\wedge t})+\int_0^{\mu\wedge t}  e^{-r s}\Big(  rV_*(Z_s)  -  \vec{D}V_* (Z_s; AZ_s) \Big)ds ].
\end{align*}
By construction, we have that $r V_*(Z_s)- \vec{D}V_* (Z_s; A Z_s) \geq  0$ on $\{s< \mu\}$. 
This implies that for all $t$:
\begin{align*}
V_*(p,y)&\geq \mathbb{E}_{\PP'_{p,q}}\left[e^{-r(\mu\wedge t)} V_*(Z_{\mu\wedge t})\right]=\mathbb{E}_{\PP'_{p,q}}\left[\indic_{t < \mu} e^{-rt} V_*(Z_{t})+ \indic_{\mu \leq t} e^{-r\mu} V_*(Z_\mu)\right]\\
&\geq  \mathbb{E}_{\PP'_{p,q}}\left[e^{-r t} f_*(Z_t) \indic_{t <\mu}+e^{-r \mu} h_*(Z_\mu) \indic_{\mu \leq t}\right]\\
&= \mathbb{E}_{\PP'_{p,q}}\left[e^{-r t} f_*(X_{t},\xi_t) \indic_{t <\mu}+e^{-r \mu} h_*(X_{\mu},\xi_\mu) \indic_{\mu \leq t}\right],
\end{align*}
where the last equality follows by conditioning with respect to $\CH^\mu_{t}$ and $\CH^\mu_\mu$.
Finally, we obtain
\begin{equation} \label{eq_dualopt}
\begin{array}{rcl}
V_*(p,y)&\geq&  \sup_{t\in \RR_+} \mathbb{E}_{\PP'_{p,q}}\left[e^{-r t} f_*(X_{t},\xi_t) \indic_{t <\mu}+e^{-r \mu} h_*(X_{\mu},\xi_\mu) \indic_{\mu \leq t}\right] \\
& = & \sup_{\tau \in \CT^Y} \mathbb{E}_{\PP'_{p,q}}\left[e^{-r \tau} f_*(X_{\tau},\xi_\tau) \indic_{\tau <\mu}+e^{-r \mu} h_*(X_{\mu},\xi_\mu) \indic_{\mu \leq \tau}\right],
\end{array}
\end{equation}
where the last equality follows from the fact that the expression inside the expectation is independent of $\CF^Y_\infty$ (and using the same method as for Remark \ref{purifystoppingtimes}).
\p
To conclude the proof, we use duality exactly as in Proposition \ref{Vsubsolbothsides2}.
\p
Recall that since $y \in \partial_2^- V(p,q)$, we have 
\[ V(p,q)= \langle q,y \rangle - V_*(p,y) .\]
It follows that 
\[ V(p,q) \leq \inf_{\tau \in \CT^Y} \langle q,y \rangle -\mathbb{E}_{\PP'_{p,q}}\left[e^{-r \tau} f_*(X_{\tau},\xi_\tau) \indic_{\tau <\mu}+e^{-r \mu} h_*(X_{\mu},\xi_\mu) \indic_{\mu \leq \tau}\right].\]
\p
For any $\tau \in \CT^Y$, we have, using that $Y$ and $\xi$ are independent
\[ \langle q,y \rangle= \EE_{p,q}[e^{-r \tau}\langle \delta_{Y_\tau}, \xi_{\tau} \rangle \indic_{\tau <\mu} + e^{-r \mu}\langle \delta_{Y_\mu}, \xi_\mu \rangle \indic_{\mu \leq \tau}].\]
We deduce that 
\begin{multline}
\langle q,y \rangle -\mathbb{E}_{\PP'_{p,q}}\left[e^{-r \tau} f_*(X_{\tau},\xi_\tau) \indic_{\tau <\mu}+e^{-r \mu} h_*(X_{\mu},\xi_\mu) \indic_{\mu \leq \tau}\right]\\
 = \EE_{\PP'_{p,q}}[e^{-r \tau}(\langle \delta_{Y_\tau}, \xi_{\tau} \rangle -f_*(X_{\tau},\xi_\tau))\indic_{\tau <\mu} + e^{-r \mu}(\langle \delta_{Y_\mu}, \xi_\mu \rangle-h_*(X_{\mu},\xi_\mu) )\indic_{\mu \leq \tau}].
 \end{multline}
It follows from the definitions of $f_*,h_*$ that:
\[ \langle \delta_{Y_\tau}, \xi_{\tau} \rangle -f_*(X_{\tau},\xi_\tau) \leq f(X_\tau,Y_\tau)\]
\[ \langle \delta_{Y_\mu}, \xi_{\mu} \rangle -h_*(X_{\mu},\xi_\mu) \leq h(X_\mu,Y_\mu),\]
from which we conclude that
\begin{multline}
\langle q,y \rangle -\mathbb{E}_{\PP'_{p,q}}\left[e^{-r \tau} f_*(X_{\tau},\xi_\tau) \indic_{\tau <\mu}+e^{-r \mu} h_*(X_{\mu},\xi_\mu) \indic_{\mu \leq \tau}\right]\\
 \leq \EE_{\PP'_{p,q}}[e^{-r \tau}f(X_\tau,Y_\tau)\indic_{\tau <\mu} + e^{-r \mu}h(X_\mu,Y_\mu)\indic_{\mu \leq \tau}].
 \end{multline}
Substituting in the previous inequality, we have 
\[ V(p,q) \leq \inf_{\tau \in \CT^Y} \EE_{\PP'_{p,q}}[e^{-r \tau}f(X_\tau,Y_\tau)\indic_{\tau <\mu} + e^{-r \mu}h(X_\mu,Y_\mu)\indic_{\mu \leq \tau}]=\inf_{\tau \in \CT^Y} \EE_{\PP_{p,q}}[\bar{J}(\mu,\tau)],\]
which proves the optimality of $\mu$. 
\end{proof}

\section{Two examples}\label{sectionexample}

{As explained in section \ref{sectionopen}, we conjecture that Theorem \ref{suffcientbothsides} can be applied to any instance of the model studied in this paper. However, the proof of such a conjecture seems very difficult and related to a stratification problem which is beyond the scope of this paper. 
For this reason, in order to illustrate how to apply our verification result, we only consider two simple examples for which we can provide explicit solutions. In both cases, we can compute the value function, and this makes possible to verify the assumptions of Theorem \ref{suffcientbothsides}. These examples are reminiscent of several explicit computations for the asymptotic value of repeated games with incomplete information on both sides by Mertens and Zamir in \cite{MZ}. As in \cite{MZ}, we do not know exactly for which class of games such an explicit computation is possible, and we only claim that the same method can be applied to some variants of the first example, or for any one-dimensional model as the second example studied below.}

\subsection{Asymmetric Information with constant states}

We assume that $K=L=\{0,1\}$ and that $R=Q=0$, which means that $X_t=X_0$ and $Y_t=Y_0$ almost surely for all $t \geq 0$.  
\p
With an abuse of notation, we write $V(p,q)$ for $V((p,1-p),(q,1-q))$ for $(p,q)\in [0,1]^2$ (and similarly for $f,h$). $V$ as well as $f$ and $h$ are thus seen as a functions defined on $[0,1]^2$. 
\p
Let us consider the particular case 
\[ h(p,q)=3p+2q-4 , \quad f(p,q)=2p+3q-1.\]
More general cases can be solved with similar arguments.
\subsubsection{The value function}

{Note at first that the fact that $R=Q=0$ implies that the directional derivatives appearing in the variational characterization of the value function given in Theorem \ref{thmextremal} are all equal to zero. Moreover, using that $r>0$, the variational inequalities do not depend on the discount factor.
Precisely,} applying Theorem \ref{thmextremal}, we have that the value function $V$ is the unique Lipschitz saddle function such that $f\geq V \geq h$ and
\be \label{cond1ex1} \forall q\in [0,1], \forall p \in Ext(V(.,q)), \; V(p,q)>h(p,q) \Rightarrow  V(p,q)\leq 0,\ee
\be\label{cond2ex1} \forall p\in [0,1], \forall q \in Ext(V(p,.)), \; V(p,q)<f(p,q) \Rightarrow  V(p,q) \geq 0.\ee
The explicit expression for $V$ is given in the next proposition. 
\begin{proposition} 
\[ V(p,q)=\left\{ \begin{array}{rcl} 0 & \text{if} & (p,q)\in [\frac{1}{2},1]\times [0,\frac{1}{2}] \vspace{.2cm} \\ \frac{2q-1}{q}(p+q-1) &\text{if} & p\geq 1-q, \text{and} \; q\geq\frac{1}{2}  \vspace{.2cm} \\ \frac{1-2p}{1-p}(p+q-1) &\text{if} & q\leq 1-p, \text{and} \; p\leq\frac{1}{2}.\end{array} \right. \]
\end{proposition}
One may prove the above proposition by verification, by checking that $V$ fulfills all the conditions given in Theorem \ref{thmextremal}. Nevertheless, let us explain briefly how to obtain directly this explicit expression.
\eqref{cond1ex1} implies that for any given $q\in [0,1]$, the map $p \rightarrow V(p,q)$ is  affine on any interval on which it is non-negative, and \eqref{cond2ex1} implies a dual property for $q\rightarrow V(p,q)$. Using these properties, it is easy to check that 
\[ V(p,0)=\left\{ \begin{array}{rcl} f(p,0) & \text{if} & p\leq \frac{1}{2} \vspace{.2cm} \\ 0 &\text{if} & p\geq \frac{1}{2}  \end{array} \right., \; V(p,1)=ph(1,1),\]
\[ V(1,q)=\left\{ \begin{array}{rcl} 0 & \text{if} & q\leq \frac{1}{2} \vspace{.2cm}\\ h(1,q) &\text{if} & q\geq \frac{1}{2}  \end{array} \right., \;V(0,q)=(1-q)f(0,0).  \]
Using that $V$ is a saddle function, we deduce also that
\[ \forall (p,q)\in \left[\tfrac{1}{2},1\right]\times \left[0,\tfrac{1}{2}\right], \; V(p,q)=0.\]
Using continuity and concavity with respect to the first variable, there exists necessarily a one-to-one map $p \in [0,1/2] \rightarrow m(p)$ such that $V(p,m(p))=0$ (it may be any closed set using only continuity of $V$, but concavity implies that it is indeed a curve). 
Using the above-mentioned properties, the following functions are affine
\[ \forall q \in [1/2,1], p \in [m^{-1}(q),1] \rightarrow V(p,q),\]
\[ \forall p \in [0,1/2], q \in [0,m(p)] \rightarrow V(p,q).\]
Writing the condition that $V$ is a saddle function at points $(p,m(p))$, it leads to $m(p)=1-p$. 

\begin{remark}
If one slightly perturbs  the coefficients of the functions $f,h$, the above method still applies. However, the expression of the map $m$ will be more complex.
\end{remark}

\subsubsection{Non-existence of the value with non-randomized stopping times}

{Let us take benefit here from having at hand an explicit solution for the value $V$ to show that in general the value of the same game where both players are restricted to pure (i.e. non-randomized) stopping times does not exist. We choose this example for convenience as the proof is very short. However, let us emphasize that one may construct a simpler game\footnote{{For example, if $K=\{0,1\}$, $R=0$, $L=\{\emptyset\}$, $f(0)=2,f(1)=-1,h(0)=1,h(1)=-2$, one can check easily that the value of the game with pure stopping times does not exists  when $p=(\frac{1}{2},\frac{1}{2})$, and that the lower value of this game is not concave.}} with incomplete information on one side only and with a constant Markov chain (i.e. with $L=\{\emptyset\}$ and $R=0$) for which this value does not exist. Note that the value neither exists in the second example studied below, but the proof of this result is much longer and more technical than for the present example.}
 
Let us define $\hat{V}^+,\hat{V}^-$ as the lower and upper value of the same game where both players are restricted to pure (classical) stopping times, i.e. for all $(p,q) \in [0,1]^2$ (using the same convention as above)
\be 
 \hat{V}^+(p,q):= \inf_{\nu \in  \CT^Y}\; \sup_{\mu \in  \CT^X}\; \mathbb{E}_{p,q}\left [ \bar{J}(\mu,\nu) \right],
\ee
\be \hat{V}^-(p,q):=  \sup_{\mu \in  \CT^X}\; \inf_{\nu \in  \CT^Y}\;  \mathbb{E}_{p,q}\left [  \bar{J}(\mu,\nu) \right].
\ee
Then a direct computation shows that:
\begin{align*}
\hat{V}^-(p,q)&=\sup_{t_0,t_1 \in \overline{\RR}_+} \; \inf_{s_0,s_1 \in \overline{\RR}_+}\; \sum_{i,j=0,1} p_iq_j e^{-r(t_i\wedge s_j)}\left(h(i,j)\indic_{t_i<s_j} +f(i,j)\indic_{s_j \leq t_j} \right)\\
&= \left\{ \begin{matrix} pq-(1-q) & \text{if} & pq>(1-q) \\ 0 &\text{if} & p \geq 1/2 \; \text{and} \; pq \leq 1-q \\ (1-2p)(pq-(1-q))&\text{if} & p < 1/2 \; \text{and} \; pq \leq 1-q  \end{matrix} \right.
\end{align*}
where $p_0=p$ and $p_1=1-p$, and similarly for $q_0,q_1$. 
On the other hand, by symmetry of the problem, one shows that:
\[ \hat{V}^+(p,q)=-\hat{V}^-(1-q,1-p).\]
The reader may check that $p \rightarrow \hat{V}^-(p,2/3)$ is not concave, that $q\rightarrow \hat{V}^+(1/3,q)$ is not convex and that 
\[ \forall p\in (0,1/2), \;\hat{V}^+(p,1-p)>V(p,1-p)>\hat{V}^-(p,1-p) \]
 which proves that the value does not exist in general when players are restricted to pure stopping times.
\subsubsection{Optimal strategies}
We will show how to construct for our example an optimal stopping time $\mu\in\mathcal{T}^X_m$ for player 1. Optimal stopping times for player 2 can be determined in a similar way. 
\p
For $(p,y)\in [0,1]\times \RR$, we consider the (restricted) convex conjugate  
\[V_*(p,y):=V_*((p,1-p),(y,0))=\max_{q \in [0,1]} qy - V(p,q). \]
As $V_*((p,1-p),(y_1,y_2))=y_2+V_*((p,1-p),(y_1-y_2,0))$, Theorem \ref{suffcientbothsides} can be applied without any modification to the functions $V$ and $V_*$ as redefined here.
Note that the set $\CH$ is given by
\[\mathcal{H}:=\{(p,y)\in [0,1]\times \RR \,|\, rV_*(p,y) - \vec{D}_2V_*(p,y; ry) \geq 0 \}.\]
In order to calculate $V_*$ we distinguish different zones of $[0,1]\times\mathbb{R}$. 

\begin{minipage}[b]{0.65\linewidth}
We set

\begin{align*}
A&=\{(p,y): p\geq 1/2 \; ,\; y \leq 0 \}\\
B&=\{(p,y): p\geq 1/2\; ,\; y\in [0,4p-2] \}\\
C&=\{(p,y): p\leq 1/2\; ,\; y\leq \tfrac{1-2p}{1-p} \}\\
D&=\{(p,y): p\geq 1/2\;,\; y \in [4p-2,1+p]\}\\
&\qquad  \cup\{(p,y) : p\leq 1/2 \;,\;y \in [\tfrac{1-2p}{1-p},1+p]\}\\
E&=\{(p,y): y \geq 1+p \}.
\end{align*}
A direct computation yields
\[ V_*(p,y)=\left\{ \begin{array}{rccl} 0 & \text{for} & (p,y)\in A \\
	 \frac{1}{2}y  & \text{for} & (p,y)\in B \\
 	1-2p  & \text{for} & (p,y)\in C \\
	 -2 \sqrt{2-y}\sqrt{1-p}+3-2p  & \text{for} & (p,y)\in D \\ 
	  y-p  & \text{for} & (p,y)\in E, \end{array} \right. \]

\end{minipage}
\hfill
\begin{minipage}[b]{0.35\linewidth}
\p
\centering
\includegraphics[scale=0.20]{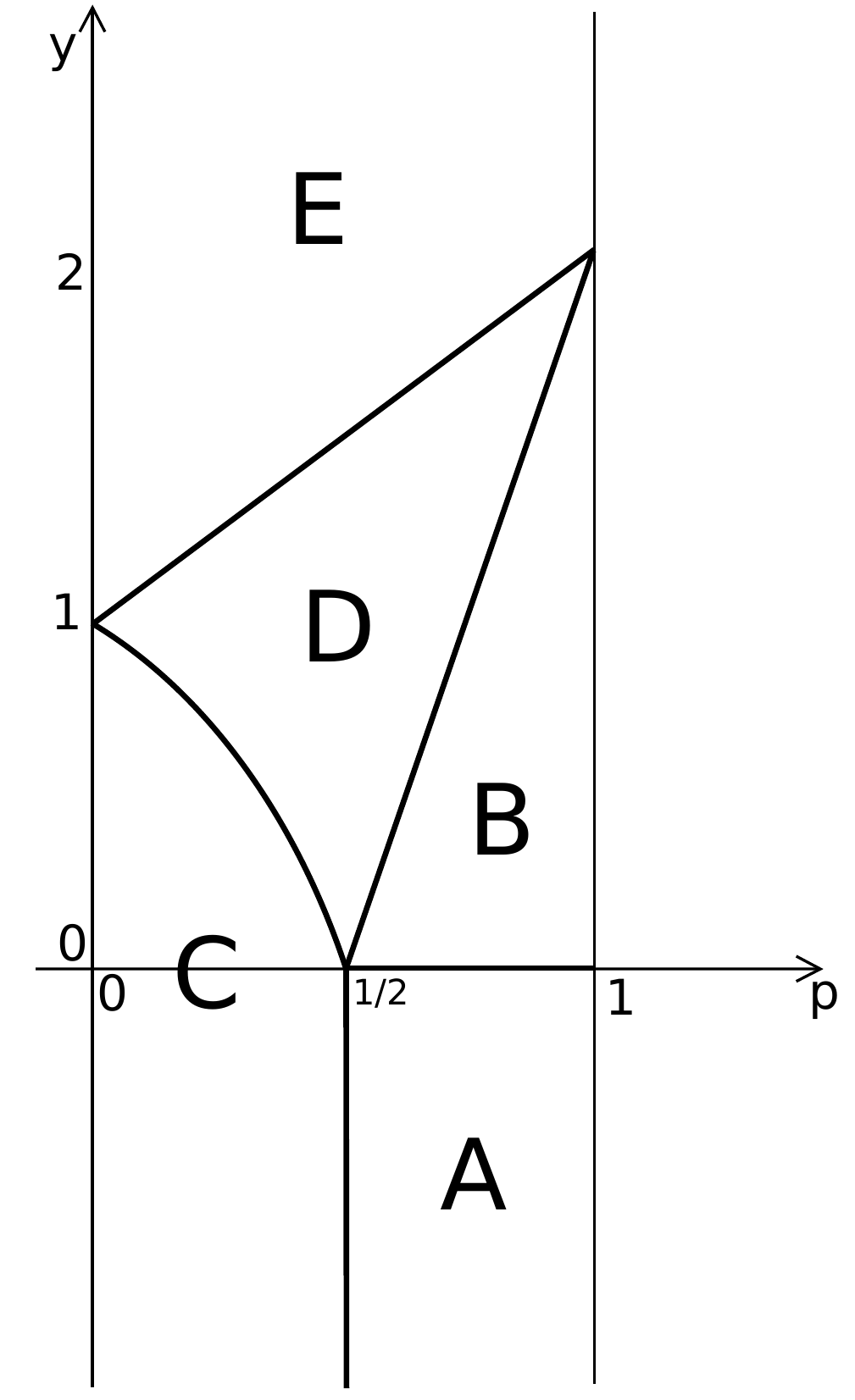}
\end{minipage}
\p
and that
\[ h_*(p,y)=\left\{ \begin{array}{rcl} 4-3p & \text{if} & y \leq 2 \vspace{.2cm} \\ y+2-3p &\text{if} & y \geq 2.\end{array} \right. .\]
It follows that $\mathcal{S}=\{ (p,y) \,|\, V_*(p,y)=h_*(p,y)\}=\{ (1,y) \,|\, y\geq 2\}$. One can easily check that 
\[  \CH=A\cup B \cup C \cup \{p=0\}.\]
Furthermore we note that $V_*$ is affine on each zone $A,B,C,E$.
\p
On $D$, $V_*$ is not an affine function, however it is easily seen that for each $p \in [0,1/2]$, $V_*$ is affine on the segment joining the point $(p,\frac{1-2p}{1-p})$ and the point $(1,2)$.
In order to construct an optimal stopping time we shall first construct some characteristics $(\alpha,\lambda,\phi)$ satisfying the conditions $(SC)$ of Definition \ref{structure}.

\p
The actual state space of our PDMP $Z$ will be $E=E_\CH \cup \CS$ with $E_\CH=A\cup C \cup \{p=0\}$. 
\p
Roughly speaking, when starting in $E_\CH$, this process has to stay in $E_ \CH$ and may jump one time over a flat part of the graph of $V_*$ from $E_\CH$ to $E_\CS=\mathcal{S}$. The condition \eqref{structuredynamic} implies that for all $z=(p,y) \in E_\CH$: 
\be \label{structureexample1} \lambda(z)(\phi(z)-z)+\alpha(z)= (0,ry) .\ee 
As $E_\CH$ has to be invariant for the flow associated to the vector field $\alpha$, if for some point $z$, the solution of the differential equation associated to the vector field $(0,ry)$ starting at $z$ does not exit $E_\CH$ immediately, we simply choose $\alpha(z)=(0,ry)$ and $\lambda(z)=0$. 
\p
We are left to consider the points  $z=(p,\frac{1-2p}{1-p})$ for $p \in (0,1/2)$. For these points, the process has to jump with positive intensity, otherwise it will exit from $E_\CH$. In order for the jump to be on a flat part of the graph of $V_*$, the only possible ending point in $\mathcal{S}$ is $(1,2)$. We define therefore $\phi(z)=(1,2)$. 
The condition \eqref{conederivative} requires  $V_*$ to be differentiable at $z$ in the direction of the cone generated by $\{\alpha(z),\lambda(z)((1,2)-z)\}$. Together with the fact that the solution of $w'=\alpha(w)$ starting at $z$ has to stay in $E_\CH$, this implies that $\alpha(z)$ is tangent to the boundary of $E_\CH$, i.e. $\alpha(z)= \kappa(z) (1,\frac{-1}{(1-p)^2})$ for some constant $\kappa(z)$.  
\p
Equation \ref{structureexample1} becomes:
\[ \lambda(z)((1,2)-z) + \kappa(z) (1,\frac{-1}{(1-p)^2}) =(0,ry),\]
which leads to the unique solution 
\[ \lambda(z)=\frac{r(1-2p)}{2}, \;\,\alpha(z)=\frac{r(1-2p)}{2}\left(-(1-p), \frac{1}{1-p}\right)=\left(-\frac{r}{2}(1-2p)(1-p),\frac{r}{2}y \right).\]
Below, we describe the optimal stopping time $\mu$ for all the possible pairs $(p,y)$ given by Theorem \ref{suffcientbothsides}. 
\p
Let $(p,y) \in \Delta(K)\times \RR$. 
\p
$\bullet$ If $(p,y)\in \{y \leq 0\} \cup\{p=0\}$, then $\mu=+\infty$ is optimal.
\p\p
$\bullet$ If $(p,y) \in \mathcal{S}$, then $\mu=0$ is optimal.
\p\p
$\bullet$ If $(p,y) \in E \cap \{0<p<1\}$, then an optimal $\mu$ is given by:
\[\mu= 0.\indic_{X_0=0} +(+\infty).\indic_{X_0=1}.\] 
The induced belief process is $\pi_t= 1 .\indic_{X_0=0}+ 0 .\indic_{X_0=1}$ for $t\geq 0$, and $\xi$ is  defined by: 
\[ \forall t\geq 0, \; \xi_t= e^{rt }\xi_0, \; \text{with} \; \xi_0=2. \indic_{X_0=0}+ \frac{y-2p}{1-p} \indic_{X_0=1}.\]
\p\p
$\bullet$ If $p \in (0, \frac{1}{2})$ and $y\in (0,\frac{1-2p}{1-p}]$, let  $z_t=(p_t,y_t)$ denote the unique solution of $\frac{d z_t}{d t}=\alpha(z_t)$ with initial condition $z_{0}=(p,y)$.  Then the optimal $\mu$ is given by 
\[ \mu(\omega,u)= (+\infty) \indic_{X_0=1} + S(u)\indic_{X_0=0},\]
where $S$ is the unique non-decreasing function from $(0,1)$ to $\RR$ such that
\[ \forall t\geq 0, \int_0^1 \indic_{S(u)>t}du = \exp\left(-\int_0^t \rho(z_s) ds\right),\]
with $\rho(z_s)=\frac{1}{p_s}\lambda(p_s,y_s)$.
The induced belief process is given by: 
\[ \forall t\geq 0, \; \pi_t = p_t\indic_{t< \mu}+ 1.\indic_{t \geq \mu},\]
and $\xi$ is defined by: 
\[\forall t\geq 0, \; \xi_t= y_t\indic_{t< \mu}+  2e^{r(t-\mu) }\indic_{t \geq \mu} .\]
Player $1$ simply stops with some intensity $\rho(z_t)=\rho(Z_t)$ conditionally on the fact that $X_0=0$. The induced intensity of jump of $Z$ is therefore equal to $\pi_t\rho(Z_t)=\lambda(Z_t)$ as required.
\p\p
$\bullet$ If $(p,y)\in B$, then an optimal $\mu$ is defined by 
\[\mu(\omega,u)= 0. \indic_{\{X_0=0, \, u\leq \frac{y}{2p}\}} + (+\infty). \indic_{\{ X_0=1\} \cup\{X_0=0, u > \frac{y}{2p}\}}.\]
It means that player $1$ stops with probability $\frac{y}{2p}$ at time $0$ conditionally on $X_0=0$, and never stops otherwise. The induced belief process is $\pi_t= \frac{2p-y}{2-y} \indic_{\mu\leq t} + 1. \indic_{\mu>t}$ for $t\geq 0$ and $\xi$ is defined by $\xi_t= 0.\indic_{\mu \leq t} + 2e^{rt} \indic_{\mu >t}$ for $t\geq 0$.
\p\p
$\bullet$ If $(p,y)\in D$, let $p' \in [0,1/2]$ be such that $(p,y)$ belongs to the segment joining $(p',\frac{1-2p'}{1-p'})$ and $(1,2)$. An easy computation shows that
\[ p'= 1 - \sqrt{\frac{1-p}{2-y}}.\]
Define $z_t=(p_t,y_t)$ as the unique solution of $\frac{d z_t}{d t}=\alpha(z_t)$ with initial condition $z_{0}=(p',\frac{1-2p'}{1-p'})$.
Then an optimal $\mu$ is defined by 
\[\mu(\omega,u)= 0. \indic_{\{X_0=0, \, u\leq x\}} + S\left(\frac{u-x}{1-x} \right) \indic_{\{X_0=0, \, u>x\}}+ (+\infty). \indic_{\{ X_0=1\}}.\]
where $x=\frac{p-p'}{p(1-p')}$ and $S$ is the unique non-decreasing function from $(0,1)$ to $\RR$ such that
\[ \forall t\geq 0, \int_0^1 \indic_{S(u)>t}du = \exp\left(-\int_0^t \rho(z_s)ds\right).\]
with $\rho_s=\frac{1}{p_s}\lambda(p_s,y_s)$. The induced belief process is given by: 
\[ \forall t\geq 0, \; \pi_t = p_t\indic_{t< \mu}+ 1.\indic_{t \geq \mu},\]
and $\xi$ is defined by: 
\[\forall t\geq 0, \; \xi_t= y_t\indic_{t< \mu}+  2e^{r(t-\mu) }\indic_{t \geq \mu} .\]
Player $1$ stops at time $0$ conditionally on  $X_0=0$ with  probability $x$, and $x$ was constructed so that $\pi_0= p' \indic_{\mu>0}+ 1. \indic_{\mu=0}$. If he did not stop at time $0$, player $1$ stops with intensity $\rho(z_t)=\rho(Z_t)$ conditionally on $X_0=0$. As above, the induced intensity of jump for $Z_t$ is $\pi_t\rho(Z_t)=\lambda(Z_t)$.


\subsection{Case of incomplete information on one side}
Next, we consider the particular case where the set $L$ is reduced to a single point, implying that player 2 has no private information. In this context, admissible stopping times of player $2$ are simply random times (see Definition \ref{mixedstoppingtimes}) and the value function $V$ depends only on $p\in \Delta(K)$.

Furthermore we assume that there are only two states for the Markov chain $X$ which is observed by player 1, i.e. $K:=\{0,1\}$. We choose $X$ to be an ergodic chain with generator $R:=\left(\begin{smallmatrix}-a & a \\ b & - b \end{smallmatrix}\right)$, where $a,b>0$ so that the unique invariant measure is $(\frac{b}{a+b},\frac{a}{a+b})$.
\p
With a slight abuse of notation, we write $V(p)$ for $V(p,1-p)$ $p\in[0,1]$ (and similarly for $f,h$). $V$ as well as $f$ and $h$ are thus seen as a functions defined on $[0,1]$. Rather than providing a complete study of this one-dimensional case, we work under the following assumptions:
\begin{itemize}
\item[(H1)]  For all $p\in [0,1]$, $0<h(p)<f(p)$. 
\item[(H2)] $f$ and $h$ are increasing on $[0,1]$, i.e. $h(0)<h(1)$ and $f(0)<f(1)$.
\end{itemize}
We impose (H1) and (H2) merely to simplify the presentation. Other cases can be studied with similar arguments.


\p
\subsubsection{Value function}

By Theorem \ref{thmextremal}, $V$ is the unique concave Lipschitz function on $[0,1]$ such that $f \geq V\geq h$ and
\begin{equation}\label{exampleeq1}
\begin{array}{rcl}
 \forall p \in [0,p^*],\;  V(p)<f(p) &\Rightarrow& \; rV(p)+((a+b)p - b)V'_+(p)  \geq 0\\  
 \forall p \in (p^*,1],\; V(p)<f(p) &\Rightarrow& \; r V(p)-((a+b)p - b)V'_-(p)  \geq 0, \\
 \ \\
\text{and for any extreme point } p\in Ext(V):\\
\ \\
p\in [0,p^*]\;\text{and}\; V(p)>h(p) &\Rightarrow& \; rV(p)+((a+b)p - b)V'_+(p)  \leq 0\\    
p\in (p^*,1]\;\text{and}\; V(p)>h(p) &\Rightarrow& \; rV(p)-((a+b)p - b)V'_-(p)  \leq 0, 
 \end{array}
\end{equation}
where $V'_+,V'_-$ denote the left and right derivative of $V$ and $p^*=\frac{b}{a+b}$. 

The following result can be shown by explicitly solving (\ref{exampleeq1}).

\begin{proposition}\label{valueexample}
Under assumptions  (H1) and (H2) we distinguish three cases:
\begin{itemize}
\item[(i)] if $\frac{b}{b+r}h(1) > f(0)$ the value function is given by 
\begin{equation}
V(p)=\begin{cases}
f(p) &\mbox{on } [0,p_0] \\
\frac{p-p_0}{1-p_0} h(1)+ \frac{1-p}{1-p_0} f(p_0)   &\mbox{on } (p_0,1],
\end{cases}
\end{equation}
where $p_0$ is the unique solution in $(0,p^*)$ of the quadratic equation: 
\begin{equation}\label{quadratic}
\frac{h(1)-f(p_0)}{1-p_0}= \frac{-rf(p_0)}{(a+b)p_0 - b}.
\end{equation}
\item  [(ii)] if  $h(0)<\frac{b}{b+r}h(1) \leq f(0)$, we have
\[ V(p)=\frac{b}{b+r}  h(1)(1-p)+ p h(1).\]
\item [(iii)] if $\frac{b}{b+r}h(1) \leq  h(0)$ we have  $V=h$ on $[0,1]$.

\end{itemize}
\end{proposition}
\p

\begin{figure}[!h]
\centering
\includegraphics[width=90mm]{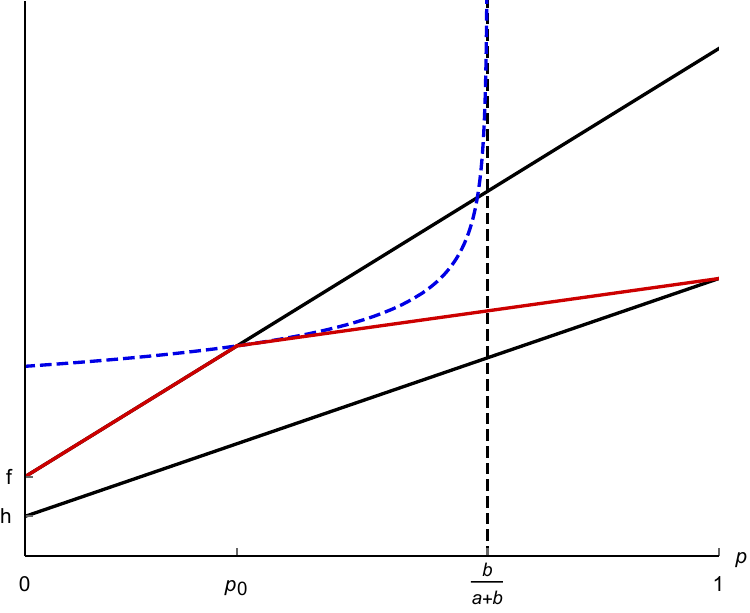}
\caption{Value function in case (i)} \label{figure2}
\end{figure}

The proof of Proposition \ref{valueexample} follows by verification by checking that \eqref{exampleeq1} holds. 
One may also easily obtain the solution by a direct approach. 
Indeed, note that all positive solutions to
\begin{equation}\label{ODE}
rw(p)+(b-(a+b)p)w'(p)=0
\end{equation}
are strictly convex, so that using $(H1)$ and \eqref{exampleeq1}, we may prove that $V(p)$ necessarily is a concave and piecewise affine function, having no extreme point in $\{h<V<f\} \cap (0,1)$. Using then $(H2)$ and \eqref{exampleeq1} applied in $p=1$ (which is always an extreme point), we deduce that $V(1)=h(1)$, and \eqref{exampleeq1} applied in $p=0$ allows to compute $V(0)$. The problem is solved in cases $(ii)$ and $(iii)$. For case $(i)$, we have $V(0)=f(0)$, and \eqref{exampleeq1} gives a condition for the left or right derivative of $V$ in the only interior extreme point $p_0$, which is \eqref{quadratic}. Figure \ref{figure2} illustrates the shape of the function $V$ (in red) in case $(i)$. $p_0$ is the unique point in $(0,p^*)$ such that the solution of \eqref{ODE} passing through $(p_0,f(p_0))$ (the blue dotted line) is tangent to the line joining $(p_0,f(p_0))$ and $(1,h(1))$.\\

{Let us give an intuition for the above result.  Recall that $p$ is the probability that the chain starts in state $0$. Therefore, $h(1)$ is the payoff of player $1$ if he stops when the chain is in state $0$. The quantity $\frac{b}{b+r}h(1)$ is the expected payoff of player $1$ when the chain starts in state $1$ if he stops at the first time the chain jumps in state $0$ while player $2$ does not stop. Therefore if $\frac{b}{b+r}h(1) \leq  h(0)$ (case $(iii)$), there is no gain to wait for the chain to jump in state $0$ for player $1$. It results that at equilibrium, player $1$ will stop immediately, independently of the position of the chain while player $2$ will never stop. If we have $h(0)<\frac{b}{b+r}h(1)$, then there is a gain for player $1$ to wait for the chain to jump in state $0$, and he will stop only when the chain is in state $0$. In case $\frac{b}{b+r}h(1) \leq f(0)$ (case $(ii)$), this gain is not sufficient for player $2$ to stop, and at equilibrium player $1$ will stop at the first time the chain reaches the state $0$ while player $2$ will never stop. If $\frac{b}{b+r}h(1) > f(0)$ (case $(i)$), the situation is more interesting. If player $2$ knows  that the chain is in state $1$, he will stop immediately to obtain the payoff $f(0)$. If player $1$ stops at the first time the chain reaches the state $0$, this strategy is completely revealing in the sense that player $2$ learns immediately that the chain is in state $1$ by observing that player $1$ did not stop at time $0$, and will stop at time $0+$ in this case. This strategy guarantees to player $1$ the payoff $(1-p)f(0)+ph(1)$ which is strictly less than the above formula given for the value. Player $1$ has therefore to take care of the beliefs induced by his strategy. His optimal strategy is formally given in the next section and can be described as follows: if the belief of player $2$ is below $p_0$, he waits until his belief reaches the value $p_0$. If the belief is above $p_0$, he stops with positive probability conditionally on the fact that $X_0=0$, and this probability is such that the belief of player $2$ given that player $1$ did not stop is equal to $p_0$. If the belief of player $2$ is equal to $p_0$, then player $1$ stops with a positive intensity $\lambda_1$ conditionally  on the fact that $X_0=0$, and $\lambda_1$ is such that the belief of player $2$ given that player $1$ did not stop remains constant over time equal to $p_0$. Knowing the value function and the strategy we just described, it is not difficult to check directly that this strategy is optimal. In the next subsection, we show how to find directly this optimal strategy by applying Theorem \ref{suffcientbothsides}.
}

\begin{remark}\label{solutionblind}
In case $(i)$, the solution $S(p)$ of the game where none of the players observe the Markov chain $X$ is easily obtained by solving the obstacle problem \eqref{pointstandard}. Precisely, we obtain that $S=f$ on $[0,p_1)$, $S=h$ on $[p_2,1]$ and that $S$ solves the ODE
\[rS(p)+((a+b)p - b)S'(p)=0 \] 
on the interval $[p_1,p_2]$ where $p_1<p_0<p_2<p^*$. The points $p_1,p_2$ are obtained using that $S$ has to be differentiable at $p_2$ (the ``smooth fit'' condition).  Note that despite the characterizations for $S$ and $V$ rely on the same variational inequalities, the convexity constraints for $V$ and the fact that the variational inequalities apply only at extreme points imply that the shape of $V$ significantly differs from that of $S$. 
\end{remark}

The solution in case $(ii)$ implies that it is optimal for player 1 to wait for the more favorable state $0$ before he stops, while the solution in case $(iii)$ implies that it is optimal for player 1 to stop immediately. The optimal strategy for the more interesting case $(i)$ is detailed in the next section.\\

\subsubsection{Optimal Strategy for the informed player}

We construct an optimal strategy for the  player 1 in case that $(H1)$ and $(H2)$ and the condition $(i)$ of Proposition \ref{valueexample} are valid.
To construct optimal stopping times we will use  Theorem \ref{suffcientbothsides}. As the parameter $q \in \Delta(L)$ plays no role here ($L$ is a singleton), we let the reader verify that Theorem \ref{suffcientbothsides} applies directly to the function $-V$ in place of $V^*$ since formally $\Delta(L)=\{1\}$ and for $y \in \RR$
\[ V^*(p,y)= y - V(p).\]
As a result, the component $\xi$ of the PDMP $Z$ plays absolutely no role, and can be replaced by $0$. In the following, we suppress the component $\xi$ and we will construct a PDMP $\pi$ having the required properties.

\p
The reader can easily verify that the set $\CH$ is given by $\CH=[0,p_0]$, where $p_0$ is given by \eqref{quadratic} and that $\mathcal{S}=\{1\}$.
\p
Let us construct the characteristics $(\alpha,\lambda,\phi)$ of the PDMP $\pi$ meeting the structure conditions $(SC)$ of Definition  \ref{structure}. The state space of $\pi$ will be $E=E_{\CH}\cup\CS$ with $E_{\CH}=\CH$.
At first the structure condition \eqref{structuredynamic} becomes
\[ \forall p \in [0,p_0], \;\lambda(p) (\phi(p)-p) +\alpha(p)= -a p+b(1-p).\] 
\p
We define $\alpha(p)=-a p+b(1-p)$ and $\lambda(p)=0$ for $p \in[0,p_0)$. For the  point $p_0$, we need to have a positive intensity of jump since the solution of $w'(t)= -a w(t)+b(1-w(t))$ starting at $p_0$ exits $\CH$ immediately. Thanks to $(SC3)$, the jump has to end in $\CS$ so that we define $\phi(p_0)=1$. The above condition becomes
\[ \alpha(p_0)=-a p_0+b(1-p_0)-\lambda(p_0)(1-p_0).\]
The condition \eqref{conederivative} requires $V$ to be differentiable in the cone generated by $\alpha(p_0)$ and $\lambda(p_0)(1-p_0)>0$. As $V$ is not differentiable at $p_0$, this implies $\alpha(p_0)\geq 0$ so that the condition becomes equivalent to the existence of a right-derivative. On the other hand, the fact that $\CH$ has to be invariant by the flow of the ODE $w'=\alpha(w)$ implies $\alpha(p_0)\leq 0$. The only possible choice is thus $\alpha(p_0)=0$ and $\lambda(p_0)=\frac{b-(a+b) p_0}{1-p_0}$.
\p
The next proposition describes optimal stopping times for all initial $p\in [0,1]$ obtained by applying Theorem \ref{suffcientbothsides}.


\begin{proposition}
Assume that $(H1)$ and $(H2)$ and the condition $(i)$ of Proposition \ref{valueexample} are valid. An optimal strategy  $\mu:\Omega\times[0,1]\rightarrow [0,+\infty)$ for player 1 is given by the following expressions.
\begin{itemize}
\item[(i)] For $p=p_0$:
\[ \mu(\omega,u):= \inf\left\{ t\geq 0 | \int_0^t \indic_{X_s=0}ds \geq \frac{- \ln (1-u)}{\lambda_1}\right\}, \;\text{with} \; \lambda_1:=\frac{b - (a+b)p_0}{p_0(1-p_0)} >0. \]
That means player $1$ stops with intensity $\lambda_1$ conditionally on  $X_t=0$.
\item[(ii)] For $p>p_0$: \[ \mu(\omega,u):= \left\{ \begin{matrix}  0 & \text{if} & X_0=0 \; \text{and} \; u\in[0, c(p)) \\ 
\inf\left\{ t\geq 0 | \int_0^t \indic_{X_s=0}ds \geq \frac{- \ln (1-\tilde{u})}{\lambda_1} \right\}  &  \text{otherwise,}    \end{matrix} \right.    \]
where $c(p):=\frac{p-p_0}{p(1-p_0)}$ and $\tilde{u}:=\frac{u-c(p)}{1-c(p)}\indic_{X_0=0}+ u \indic_{X_0=1}$.\\
That means that at time $t=0$, player $1$ stops with probability $\frac{p-p_0}{p(1-p_0)}$ if $X_0=0$, and if he did not stop at $t=0$, he stops with intensity $\lambda_1$ conditionally on the fact that $X_t=0$.
\item[(iii)]  For $p<p_0$:
\[ \mu(\omega,u):= \inf\left\{ t\geq t(p) | \int_{t(p)}^t \indic_{X_s=0}ds \geq \frac{- \ln (1-u)}{\lambda_1}\right\},    \]
where $t(p):= \inf \{t \geq 0 | \PP_p(X_t=0)=p_0\}$.\\
That means that player $1$ waits until the probability that  player 2 assigns to the event $X_t=0$ reaches $p_0$  and then stops with intensity $\lambda_1$ conditionally on the fact that $X_t=0$.
\end{itemize}
\end{proposition}

Note that the same method can be applied to compute explicitly optimal strategies for player $2$ in this example. These solutions being quite similar to the previously studied cases,  they are not presented here.

\section{Open questions}\label{sectionopen}

{The first natural question is to which class of games our verification result (Theorem \ref{suffcientbothsides}) applies. We actually conjecture that it may be applied to any instance of our model, but all the examples in which we were able to use it share the property that $V$ and $V_*$ admits a smooth stratification. Once we know that such a stratification exists, we think that the method used in section \ref{sectionexample} to construct the characteristics $(\alpha,\lambda,\phi)$ satisfying the required assumptions should work. However, we do not know how to handle the problem of existence of a smooth stratification for the value function, and such stratification problems are known to be very difficult in geometry.}

{A possible extension} is to investigate the case where the Markov processes $X$ and  $Y$ have infinite state space, e.g. diffusion processes. The main difficulty is that our approach leads formally to study partial differential equations in infinite dimensional spaces of probabilities. The only results in this direction consider differential games where the information parameters of the game do not evolve over time (see Cardaliaguet and Rainer \cite{carda12}) {and more recently a game where only one player observes a Brownian motion (see Gensbittel and Rainer \cite{obsbrown})}.

Another interesting question is whether our methods can be adapted to consider Markov chains $(X,Y)$ which are correlated. Indeed, in our proof of the Dynamic Programming principle the independence plays an important role, since it allows to detach the two dynamics in (\ref{uncorrelated}). The case of correlated information in the information asymmetry is static has been studied in \cite{Olioubarton}. Its generalization to an evolving setting is far from being obvious and an interesting subject for further studies.

In view of possible applications to stopping games arising e.g. in financial mathematics, one may also consider models with publicly observed diffusive dynamics. The particular case  where the information parameters were not evolving was considered in \cite{grunstopping}, {and the case with symmetric incomplete information on the drift parameter of a diffusion was analyzed by De Angelis, Gensbittel and Villeneuve \cite{angelis1}}. A generalization of these results is an interesting point for further research.

We only consider the values of discounted games, and a natural question is whether there exists a limit value when the discount factor $r$ vanishes. In the particular case where the Markov chains are constant (i.e. $R=0$ and $Q=0$), one sees that the value function is independent of $r$, and thus the limit value exists and is characterized by the same variational inequalities (note however that the optimal strategies depend on $r$, see the first example in section \ref{sectionexample}). In the second example of section \ref{sectionexample}, it is easily seen from Proposition \ref{valueexample} that a limit value exists, but we do not know if this phenomenon is general and if we can obtain a characterization of the limit value whenever it exists. Such a characterization was provided in Gensbittel and Renault \cite{GensbittelRenault} for the limit value of Markov chain games with incomplete information on both sides, and it relies on the distinction between hard information which tells something about the recurrence class to which the chain belongs, and soft information which may only inform about the relative position of the chain within the different recurrence classes. In a repeated game, soft information is renewable in the sense that if player $1$ reveals soft information, he can wait until the marginal distribution of the chain is close to an invariant measure within each recurrence class. 
{Such a distinction may be useful to analyze the existence and characterization of the limit value in our model when there are several recurrence classes, and this is an interesting point for further research.}

{In several models of timing games, such as strategic experimentation models (see e.g. Rosenberg, Solan and Vieille \cite{RSV} and the references therein) the players receive a flow of payoffs instead of a terminal payoff. Even if the natural framework of these problems is a non zero-sum game, one may consider by analogy a zero-sum model with integral payoffs given by 
\begin{equation*}
J^1(\mu,\nu)(\omega,u,v):= \int_0^\nu e^{-r s} f(X_{s},Y_{s})ds \indic_{\nu <\mu}+\int_0^\mu e^{-r s} h(X_{s},Y_{s})ds \indic_{\mu \leq\nu}.
\end{equation*}
The extension of our results to this model is left for future research. }

{At last, the extension of our result to non zero-sum models is a challenging and very interesting problem, which may allow to handle strategic experimentation models or real option problems with incomplete information. Although the study of Nash equilibria in non zero-sum games generally differs drastically from the study of optimal strategies in zero-sum games, it appears that in different models of non zero-sum stopping games, equilibrium strategies have some similarities with the optimal strategies of zero-sum stopping games. For example, in the recent work of De Angelis, Ferrari and Moriarty \cite{angelis2}, war of attrition type of non zero-sum games are considered, and the variational characterization of these equilibria using threshold strategies is very close to that of the corresponding zero-sum model. Ekstrom, Glover and Leniec \cite{ekstrom} consider a zero-sum stopping game with heterogeneous prior beliefs (which is equivalent to a non zero-sum stopping game) and provide an example of equilibrium in mixed stopping times, where a player stops with some varying intensity, exactly as in our Theorem \ref{suffcientbothsides}. These similarities suggest to study the problem of existence of equilibria in mixed strategies in non zero-sum stopping games with incomplete information using variational methods and strategies having the same structure as the optimal strategies described in Theorem \ref{suffcientbothsides}.
}

\p
\textbf{Acknowledgment:} The authors would like to thank the associate editor and an  anonymous referee for their pertinent comments and helpful suggestions.

\appendix
\section{Technical proofs and auxiliary tools}

\subsection{Auxiliary results of convex analysis}\label{auxiliary}

We prove here some elementary results in convex analysis that we are using in the proof of Proposition \ref{V*supersolbothsides2}. The following lemmas are easy adaptations of classical and well-known results to Lipschitz convex functions with polyhedral domains. As references covering exactly what we need were difficult to find, we decided to add this appendix for the convenience of the reader. Note that we do not try to provide the most general version of these lemmas. 
\p

\begin{lemma}\label{superdiff}
Let $f$ be a Lipschitz convex function from $C$ to $\RR$. Assume that $C$ is a polyhedron, then if $T_C(x)$ denotes the tangent cone of $C$ at $x$:
\[ \forall x\in C, \forall v\in T_C(x), \vec{D}f(x;v) =\max_{u \in \partial^- f(x)} \langle u, v \rangle .\]    
\end{lemma}
\begin{proof}
This relation holds for any point $x$ in the relative interior of $C$ (see Theorem 23.4 in \cite{rockafellar}). For $x$ in the relative boundary of $C$, the left hand-side of the above equality is always greater than the right-hand side using the definition of subgradients.  
\p
For $m \in \NN^*$ with $m \geq M$, where $M$ is the Lipschitz constant of $f$, define the Moreau-Yosida regularization 
\[ \forall x\in \RR^n, \; f_m(x):= \inf_{y \in C} f(y) +m|y-x| .\]
The function $f_m$ is convex, $m$-Lipschitz (so that subgradients of $f_m$ are uniformly bounded by $m$) and coincides with $f$ on $C$. 
As $f_m \leq f$, we have for all $x \in C$, $\partial^- f_m(x) \subset \partial^-f(x)$. For any $v \in T_C(x)$, we have that $x+tv \in C$ for all sufficiently small $t>0$ (here we use the polyhedron assumption) and therefore:
\[ \vec{D}f(x;v)= \vec{D} f_m(x;v) =   \max_{u \in \partial^- f_m(x)} \langle u,v \rangle\leq \sup_{u \in \partial^- f(x)} \langle u, v \rangle  .\]
\end{proof}

\begin{lemma}[Danskin] \label{danskinlemma}
Let $P$ be a non-empty compact subset of $\RR^m$ and $C$ a non-empty polyhedron in $\RR^n$. Let $f$ be a real-valued Lipschitz function defined on $P\times C$. Assume that for all $p \in P$, the function $f_p$ defined on $\RR^n$ by $f_p(x)=f(p,x)$ is a convex function. Define $g(x)=\sup_{p \in P}f(p,x)$. Then, if for $\bar{x} \in C$, the maximum $\max_{p \in P} f(p,\bar{x})$ is uniquely attained in $\bar{p}$, we have 
\[ \partial^- g(\bar{x}) = \partial^- f_{\bar{p}}(\bar{x}) .\] 
\end{lemma}
\begin{proof}
Note that both sets are non-empty due to the Lipschitz assumption, and that the inclusion $\partial^- f_{\bar{p}}(\bar{x}) \subset \partial^- g(\bar{x})$ is a direct consequence of the definitions. Let us prove the reverse inclusion. Assume by contradiction the there exists $v \in \partial^- g(\bar{x}) \setminus \partial^- f_{\bar{p}}(\bar{x})$. Then, using a separation argument, there exists $z \in \RR^n$  and $\varepsilon>0$ such that: 
\be \label{eqseparation}  \forall u \in \partial^- f_{\bar{p}}(\bar{x}) , \;\langle z, v \rangle \geq \varepsilon + \langle z, u \rangle .\ee
For all $u \in \partial^- f_{\bar{p}}(\bar{x}) $, $w$ in the normal cone of $C$ at $\bar{x}$ and $j \in \NN^*$, we have $u+j w   \in \partial^- f_{\bar{p}}(\bar{x}) $. Replacing $u$ bu $u+jw$ in \eqref{eqseparation}, and taking the limit as $j\rightarrow \infty$, we deduce that $\langle z,w \rangle \leq 0$, implying that  $z$ belongs to the tangent cone of $C$ at $\bar{x}$. Since $C$ is a polyhedron, $\bar{x}+tz \in C$ for all $t \in (0, \alpha)$ for some $\alpha >0$. Up to replace $z$ by $\alpha z$, we may assume that $\bar{x}+z \in C$.  Let $p_n$ be  a sequence maximizing  $f(p,\bar{x}+ \frac{z}{n})$. The sequence  $p_n$ converges to $\bar{p}$ using the continuity of $f$ and $g$. For $s \in (0,1)$ and $n$ such that $\frac{1}{n} \leq s$ we have 
\[ \frac{f(p_n,\bar{x}+sz)- f(p_n,\bar{x})}{s} \geq \frac{f(p_n,\bar{x}+\frac{z}{n})- f(p_n,\bar{x})}{n^{-1}} \geq \frac{g(\bar{x}+\frac{z}{n})-g(\bar{x})}{n^{-1}} \geq \langle z,v \rangle .\]
Letting $n$ go to $+\infty$, we deduce that for all $s\in(0,1)$, 
\[ \frac{f( \bar{p},\bar{x}+sz)- f(\bar{p},\bar{x})}{s} \geq \langle z,v \rangle.\]
Taking the limit when $s\rightarrow 0+$ and using the preceding lemma, we conclude that
\[ Df_{\bar{p}}(\bar{x},z)=\sup_{u \in \partial^- f_{\bar{p}}(\bar{x})} \langle z,u \rangle \geq \langle z, v \rangle,\]
which is a contradiction.
\end{proof}

Before stating the next lemma, let us recall the definition of an exposed point.
\begin{definition}\label{exposeddef}
Let $f:\RR^m \rightarrow \mathbb{R}\cup\{+\infty\}$ be a convex function. The set of exposed points of $f$ is defined is defined as the set of all $x \in \RR^n$ such that there exists $x^* \in \partial^-f(x)$ such that 
\[ \forall x'\neq x,  f(x') > f(x)+\langle x^*,x'-x\rangle.\]
\end{definition}

\begin{lemma}\label{exposed}
Let $C \subset \RR^m$ be a compact polyhedron and $f: C \rightarrow \RR$ a convex Lipschitz function. If $x$ is an extreme point of $f$ and $z$ in the tangent cone of $C$ at $x$, then there exists a sequence $x_n$ of exposed points of $f$ with limit $x$ such that
\[ \vec{D}f(x_n;z) \rightarrow \vec{D}f(x;z). \]
\end{lemma}
\begin{proof}
Let us choose $y \in \partial^-f(x)$ such that $\vec{D}f(x;z)=\langle z, y\rangle$. We claim that $x$ is an extreme point of $\partial^- f^*(y)$. Note at first that Fenchel's Lemma implies that $x\in \partial^- f^*(y)$. Assume then that $x$ is not an extreme point of $\partial^- f^*(y)$, so that there exists a segment $(x_1,x_2)$ containing $x$ and included in $\partial^- f^*(y)$. It follows that $y \in \partial^- f(x')$ for all $x' \in(x_1,x_2)$, and thus that $f$ restricted to this segment is affine, which contradicts the fact that $x$ is an extreme point of $f$. 
\p
Using Theorem 25.6 of \cite{rockafellar}, there exists therefore a sequence $y_n$ with limit $y$ such that $f^*$ is differentiable at $y_n$ and the sequence $x_n:=\nabla f^*(y_n)$ has limit $x$. (In the proof of Theorem 25.6 of \cite{rockafellar}, this statement is proved only for exposed points of $\partial^- f^*(y)$, but this extends easily to extreme points by a diagonal extraction). The sequence of points $x_n$ is made of exposed points of $f$ by Corollary 25.1.2 in \cite{rockafellar} and 
\[ \liminf_n\vec{D}f(x_n;z)\geq \lim_n\langle y_n, z \rangle =  \langle y,z \rangle =\vec{D}f(x;z).\]
The reverse inequality follows from Lemma \ref{superdiff} together with the fact the correspondence of superdifferentials has closed graph.  
\end{proof}


\begin{thebibliography} {99}


\bibitem{aumann} R.J. Aumann, \emph{Mixed and behavior strategies in infinite extensive games}, in Advances in Game Theory, Ann. Math. Stud. 52, M. Dresher, L.S. Shapley, and A.W. Tucker, eds., Princeton University Press, Princeton, NJ, 1964.
(1964),




\bibitem{Bremaud} P. Bremaud, \emph{Point Processes and Queues: Martingale Dynamics}, Springer Series in Statistics. Springer-Verlag, New York-Berlin, 1981.


\bibitem{Baxter} J.R. Baxter and R.V. Chacon, \emph{Compactness of stopping times}, Zeitschrift f{\"u}r Wahrscheinlichkeitstheorie und Verwandte Gebiete, 1977, \textbf{40}, pp. 169--181.




\bibitem{cardadiff} P. Cardaliaguet, {\it Differential games with asymmetric information.},  SIAM J. Control Optim., 2007, \textbf{46}, pp. 816--838

\bibitem{cardadouble}
P. Cardaliaguet, \emph{A double obstacle problem arising in differential game theory}, Journal of Mathematical Analysis and Applications, 2009, \textbf{360}, pp. 95--107.

\bibitem{cardastochdiff}
P. Cardaliaguet and C. Rainer, \emph{Stochastic Differential Games with Asymmetric Information}, Applied Mathematics and Optimization, 2009, \textbf{59}, pp. 1--36.

\bibitem{cardaexemple}
P. Cardaliaguet and C. Rainer, \emph{On a Continuous-Time Game with Incomplete Information}, Math. of Oper. Res., 2009, \textbf{34}, pp. 769--794.

\bibitem{carda12}
P. Cardaliaguet and C. Rainer, \emph{Games with incomplete information in continuous time and for continuous types},
Dynamic Games and Applications, 2012, \textbf{2}, pp. 206--227.

\bibitem{cardaetal}
P. Cardaliaguet, C. Rainer, D. Rosenberg, N. Vieille, \emph{Markov games with frequent actions and incomplete information}, Math. of Oper. Res., 2016, pp. 49--71

\bibitem{Davis} MHA, Davis, \emph{Piecewise-deterministic Markov processes: A general class of non-diffusion stochastic models}, Journal of the Royal Statistical Society. Series B (Methodological), 1984, \textbf{46}, pp. 353--388

\bibitem{angelis1} T. De Angelis, F. Gensbittel, and S. Villeneuve, \emph{A Dynkin game on assets with incomplete information on the return}, preprint, 2017,  \href{url: http://arxiv.org/abs/1705.07352}{arxiv.org/abs/1705.07352} 

 
 
\bibitem{angelis2} T. De Angelis, G. Feraari, and J. Moriarty, \emph{Nash equilibria of threshold type for two-player nonzero-sum games of stopping}, preprint, 2017, \href{url:http://arxiv.org/abs/1508.03989}{arxiv.org/abs/1508.03989}

\bibitem{dellacheriemeyer}
C. Dellacherie, J.-P. Meyer, \emph{Probabilities and potential, A}. North-Holland Mathematics Studies, 29. North-Holland Publishing Co., Amsterdam-New York, 1978. 

\bibitem{demeyerdual}
B. De Meyer, \emph{Repeated Games, Duality and the Central Limit Theorem}, Math. of Oper. Res., 1996, \textbf{21}, pp. 237--251

\bibitem{demeyergeb} 
B. De Meyer, \emph{Price dynamics on a stock market with asymmetric information}, Games and Econ. Behav., 2010, \textbf{69}, pp. 42--71.



\bibitem{ekstrom} E. Ekstr\"{o}m, K. Glover, and M. Leniec, \emph{Dynkin games with heterogeneous beliefs}, Journal of Applied Probability, 2017, \textbf{54}, pp. 236--251. 


\bibitem{eckstrom} E. Ekstr\"om, G. Peskir, \emph{Optimal stopping games for Markov processes.} SIAM J. Control Optim., 2008, \textbf{47}, pp. 684--702.


\bibitem{EvansSouganidis} L. C. Evans and P. Souganidis, \emph{Differential games and representation formulas for solutions of Hamilton-Jacobi-Isaacs equations}, Indiana U. Math. J., 1984, \textbf{33}, pp. 773--797

\bibitem{fan}
K. Fan, \emph{Minimax theorems}, Proc. Nat. Acad. Sci. U.S.A., 1953, \textbf{39}, pp. 42--47.

\bibitem{friedman} A. Friedman, \emph{Stochastic games and variational inequalities}, J Archive for Rational Mechanics and Analysis, 1973, \textbf{51}, pp. 321--346.

\bibitem{fabiencavu} F. Gensbittel, \emph{Extensions of the Cav(u) theorem for repeated games with one-sided information}, Math. of Oper. Res., 2015, \textbf{40}, pp. 80--104. 

\bibitem{fabiencovariance}
F. Gensbittel, \emph{Covariance control problems over martingale arising from game theory}, SIAM J. Control Optim., 2013, \textbf{51}, 1152--1185.

\bibitem{gensbittel2013}
F. Gensbittel, \emph{Continuous-time limit of dynamic games with incomplete information and a more informed player}, Int. J. of Game Theory, 2016, \textbf{45}, pp. 321--352.

\bibitem{obsbrown}
F. Gensbittel and C. Rainer, \emph{A two player zero-sum game where only one player observes a Brownian motion}, Dynamic Games and Applications, 2017, doi:10.1007/s13235-017-0219-5

\bibitem{GensbittelRenault} 
F. Gensbittel and J. Renault, \emph{The value of Markov Chain Games with incomplete information on both sides}, Math of Oper. Res., 2015, \textbf{40}, pp. 820-–841.

 
\bibitem{grunstopping}
C. Gr\"{u}n, \emph{On Dynkin games with incomplete information}. SIAM J. Control Optim., 2013, \textbf{51}, pp. 4039--4065.
\bibitem{grungirsanov}
C. Gr\"{u}n, \emph{A BSDE approach to stochastic differential games with incomplete information}, Stochastic Processes and their Applications, 2012, 122, pp. 1917--1946.

\bibitem{heuer} M. Heuer, \emph{Asymptotically optimal strategies in repeated games with incomplete information}, Int. J. of Game Theory, 1992, \textbf{20}, pp. 377--392.

\bibitem{JacodSkorokhod} J.Jacod, A.V. Skorokhod, \emph{Jumping filtrations and martingales with finite, S{\'e}minaire de Probabilit{\'e}s}, 1994, \textbf{28}, pp. 21--35

\bibitem{JacodShiryaev}
J. Jacod  and A.N. Shiryaev, \emph{Limit theorems for stochastic processes, Second Edition}, Grundlehren der Mathematischen Wissenschaften, Springer-Verlag, Berlin, 2003.

\bibitem{laraki} R. Laraki, R. \emph{Variational inequalities, system of functional equations, and incomplete information repeated games}, SIAM J. Control Optim., 2001,  \textbf{40}, pp. 516-524.

\bibitem{larakisolan}  R. Laraki and E. Solan, \emph{The value of zero-sum stopping games in continuous time.} SIAM J. Control Optim., 2005, \textbf{43}, pp. 1913--1922

\bibitem{LepeltierMaingueneau}
Lepeltier, J.P., Maingueneau, M.A., \emph{Le jeu de Dynkin en th{\'e}orie g{\'e}n{\'e}rale sans l'hypoth{\`e}se de Mokobodski},
Stochastics, 1984, {\bf 13}, pp. 25--44.

\bibitem{MZ} J.-F. Mertens and S. Zamir, \emph{The value of two-person zero-sum repeated games with incomplete information}, Int. J. of Game Theory, 1971, \textbf{1}, pp. 39--64 .


\bibitem{meyer}
P.A. Meyer, \emph{Convergence faible et compacit\'{e} des temps d'arr\^{e}t, d'apr\`{e}s Baxter et Chacon}, S\'{e}minaire de probabilit\'{e}s de Strasbourg,  1978, \textbf{12}, pp. 411--423. 



\bibitem{neyman} A. Neyman, \emph{Existence of optimal strategies in Markov games with incomplete information}, Int. J. of Game Theory, 2008,  \textbf{37}, pp. 581--596.

\bibitem{Olioubarton} M. Oliu-Barton, \emph{Differential Games with Asymmetric and Correlated Information}, Dynamic Games and Applications, 2015, \textbf{31}, pp. 490--512.

\bibitem{Renault2006}  J. Renault,  \emph{The value of Markov chain games with lack of information on one side}, Mathematics of Operations Research, 2006, \textbf{31}, pp. 490--512.

\bibitem{RSV} D. Rosenberg, E. Solan, and N. Vieille, \emph{Social Learning in One-Arm Bandit Problems}, Econometrica, 2007, \textbf{75}, pp. 1591--1611

\bibitem{rosenbergsorin} D. Rosenberg and S. Sorin, \emph{An operator approach to zero-sum repeated games}, Israel Journal of Mathematics, 2001, \textbf{121}, pp. 221--246.

\bibitem{rockafellar}  R. T. Rockafellar, \emph{Convex analysis}. Princeton Mathematical Series, No. 28 Princeton University Press, Princeton, N.J. 1970. 

\bibitem{tsirelson} E. Solan, B. Tsirelson, and N. Vieille, \emph{Random stopping times in stopping problems and stopping games}, arxiv.1211.5802

\bibitem{shmayasolan} E. Shmaya and E. Solan, \emph{Equivalence between Random Stopping Times in Continuous Time}, 2014, Preprint,  arXiv:1403.7886v1 [math.PR] 

\bibitem{Sorin} S. Sorin, \emph{A First Course on Zero-Sum Repeated Games}, Springer, 2002.


\bibitem{vieilletouzi} N. Touzi and N. Vieille, \emph{Continuous-time Dynkin games with mixed strategies.}, 2002, SIAM J. Control Optim., \textbf{41}, pp. 1073--1088. 

\end{thebibliography}
\end{document}